\documentclass[10pt,reqno]{amsart}
\usepackage{geometry}
\usepackage[T1]{fontenc}
\usepackage{lmodern}
\usepackage{amsmath,amssymb,amsthm,mathtools}
\usepackage{xcolor}

\usepackage{tikz}
\usetikzlibrary{decorations.pathreplacing}

\newtheorem{theorem}{Theorem}[section]
\newtheorem{proposition}[theorem]{Proposition}
\newtheorem{lemma}[theorem]{Lemma}
\newtheorem{corollary}[theorem]{Corollary}
\theoremstyle{definition}
\newtheorem{definition}[theorem]{Definition}
\newtheorem{remark}[theorem]{Remark}

\newcommand{\R}{\mathbb{R}}
\newcommand{\Sym}{\operatorname{Sym}}
\newcommand{\GL}{\operatorname{GL}}
\newcommand{\tr}{\operatorname{tr}}
\newcommand{\diag}{\operatorname{diag}}
\newcommand{\ExpD}{\operatorname{Exp}_{\Delta}}
\newcommand{\LogD}{\operatorname{Log}_{\Delta}}
\newcommand{\gm}{\mathbin{\#}} 
\newcommand{\lv}{\mathbin{\odot_{\Delta}}} 
\newcommand{\hLie}{\mathfrak{h}} 

\newcommand{\cg}{\textnormal{\textsl{g}}} 

\usepackage{xparse}

\NewDocumentCommand{\cv}{e{_}}{%
  \mathbin{\#_{\mathrm{cv}\IfValueT{#1}{,#1}}}%
}

 \usepackage[pagebackref,  colorlinks=true,  citecolor=cyan,
 citebordercolor={0 .255 .255},
  linkbordercolor={0 .255 .255},
  linkcolor=cyan
 ]{hyperref}

 \usepackage{cite}

 \def\equationautorefname~#1\null{(#1)\null}

\title[ Sparsity-preserving means]{Sparsity-preserving means on a homogeneous convex cone }
\author{Khalid Koufany}
\address{Universit\'e de Lorraine, CNRS, IECL, F-54000 Nancy, France}
\email{khalid.koufany@univ-lorraine.fr}

 

\begin{document}

\begin{abstract}
We study two means on the five-dimensional Vinberg cone $\Omega$, viewed as a sparse cone of positive definite $3\times3$ matrices. The Cholesky-Vinberg mean is obtained from the simply transitive solvable group $H$ associated with $\Omega$. It is $H$-equivariant, and its interpolation curves are geodesics for two flat metric connections with torsion. One of the corresponding metrics is the canonical Hessian metric of the cone. The Log-Vinberg mean is defined by affine interpolation in the global logarithmic coordinates of the associated Vinberg algebra, or clan. The resulting Riemannian metric is complete and flat, and its weighted Fr\'echet means have explicit formulas. Both constructions preserve the zero pattern defining $\Omega$.
\end{abstract}

\maketitle

 
\section{Introduction}

The cone $\Sym^{++}(n,\R)$ of real symmetric positive definite matrices is both an open convex cone and the Riemannian symmetric space
$$
\Sym^{++}(n,\R)\simeq \GL(n,\R)/\mathrm O(n).
$$
It is also the interior of the cone of squares in the Euclidean Jordan algebra $\Sym(n,\R)$. These descriptions
connect matrix analysis with Lie theory and Jordan algebra methods.

One of the basic constructions   is the geometric mean. For  $x,y\in \Sym^{++}(n,\R)$, it is defined by
$$
x\gm y
=
x^{1/2}\bigl(x^{-1/2}yx^{-1/2}\bigr)^{1/2}x^{1/2}.
$$
It is a Kubo-Ando mean, the 
unique positive definite solution $z$
of $zx^{-1}z=y$, and the midpoint of the affine-invariant geodesic
$$
\gamma^{\gm}_{x,y}(t)
=
x^{1/2}\bigl(x^{-1/2}yx^{-1/2}\bigr)^t x^{1/2},
\qquad 0\le t\le 1.
$$
See \cite{PuszWoronowicz1975,KuboAndo,Bhatia2007,FiedlerPtak,Ando79,Ando87,AndoLiMathias2004,LawsonLim2001}, and \cite[Chapter~4]{Bhatia2007} for a historical account. Lim \cite{LimGM} extended this construction to arbitrary symmetric cones by means of their Euclidean Jordan algebra structure.

 The geometry of positive definite matrices is important in applications. For example, such matrices 
 represent covariance structures, diffusion tensors, and information matrices.
Their Euclidean average
need not reflect the features one wishes to preserve. For instance, the determinant of an arithmetic
mean can exceed the determinant of every input matrix, a phenomenon known as the swelling effect. This has
motivated the study of
other geometries, such as the
affine-invariant and Log-Euclidean metrics, see
\cite{Moakher2005,Pennec2006,Arsigny2007,Dryden2009,Bhatia2007}.
  Lin's Log-Cholesky metric \cite{Lin2019} provides another  complete flat geometry, defined through the Cholesky decomposition.

In the present paper, our main concern is sparsity.
Fixed zero patterns arise in many situations. In statistics, they
occur in graphical models and covariance selection, and are closely
related to positive definite completion problems; see
\cite{GroneJohnsonSaWolkowicz1984,Lauritzen1996,Scott-Tuma}.
In quantum information, vanishing off-diagonal entries of a density
matrix express the absence of the corresponding coherences in a
fixed basis; see \cite{NielsenChuang2010,Zurek2003}.
The question
addressed here is how to interpolate positive definite matrices
while preserving their prescribed zeros and retaining
useful geometric properties.

 The difficulty  already appears in a small  matrix example. Let
$$
V=
\left\{
\begin{pmatrix}
x_1&0&x_4\\
0&x_2&x_5\\
x_4&x_5&x_3
\end{pmatrix}
:
 x_1,\dots,x_5\in\R
\right\}
\subset \Sym(3,\R)
$$
and consider the open convex cone
$$
\Omega=V\cap \Sym^{++}(3,\R)
$$
  of positive definite $3\times 3$ real symmetric matrices with the fixed zero entries $x_{12}=x_{21}=0$.  
  The matrices 
$$
e=I_3,
\qquad
 y=
\begin{pmatrix}
2&0&1\\
0&2&1\\
1&1&3
\end{pmatrix}.
$$
both  belong to $\Omega$, whereas their ambient geometric mean satisfies 
$$
(e \# y)_{12}=\left(y^{1 / 2}\right)_{12}=\frac{2}{3}-\frac{\sqrt{ } 2}{2} \neq 0 .
$$
Consequently,  the usual Riemannian midpoint need not lie in the sparse cone. 

This failure is structural. The subspace $V$ is not a Jordan subalgebra of  the Euclidean Jordan algebra $\Sym(3,\R)$, so matrix square roots, Jordan powers, and affine-invariant geodesics need not preserve $\Omega$; see \cite{FarautKoranyi1994,LimGM}. 
The obstruction extends beyond the geometric mean. Indeed, the
weighted arithmetic means are the only ambient Kubo-Ando means whose
restrictions preserve   $\Omega$,  see \cite{DK}.

At the same time, the cone $\Omega$ is not arbitrary.  Ii is the five-dimensional Vinberg cone, the smallest homogeneous convex cone that is not symmetric and it is attached to a particular algebraic structure of $V$.  Homogeneous cones were studied  by Vinberg, who extended the Koecher-Vinberg correspondence for symmetric cones to the homogeneous convex cones setting by introducing \emph{compact normal left-symmetric algebras}, also called \emph{clans} or \emph{Vinberg algebras}; see \cite{Vinberg1963}. Related developments appear in the work of Koszul, Rothaus, Shima, and Ishi; see \cite{Koszul,Rothaus1966,Shima1976,Shima2007,Ishi2001,Ishi2013, Ishi2015}. 
In particular, a homogeneous convex cone admits a simply transitive split solvable group of linear automorphisms.

For the cone considered here, the simply transitive solvable group is the lower triangular group
$$
H=
\left\{
\begin{pmatrix}
a_1&0&0\\
0&a_2&0\\
u_4&u_5&a_3
\end{pmatrix}
:
 a_1,a_2,a_3>0,
\ u_4,u_5\in\R
\right\}.
$$
It acts on $\Omega$ by congruence,
$
\rho(T)x=TxT^{\top},
$
and the orbit map
$$
\Phi:H\longrightarrow \Omega,
\qquad
\Phi(T)=\rho(T)e=TT^{\top},
$$
is a global diffeomorphism.  Its inverse associates to $x\in\Omega$
its lower Cholesky factor $T_x\in H$.

The first construction is 
the \emph{Cholesky-Vinberg path}
$$
\gamma^{\mathrm{cv}}_{x,y}(t)
=
\Phi\bigl(T_x(T_x^{-1}T_y)^t\bigr)=T_x(T_x^{-1}T_y)^t{(T_x^{-1}T_y)^t}^\top T_x^\top,
\qquad 0\le t\le1,
$$
where powers are principal powers in the triangular group.  Its midpoint
$$
x\cv y:=\gamma^{\mathrm{cv}}_{x,y}(1/2)
$$
is the \emph{Cholesky-Vinberg mean}.  The path stays in $\Omega$, is equivariant under the congruence action of $H$ and satisfies
$$
\operatorname{det} \gamma_{x, y}^{\mathrm{cv}}(t)=(\operatorname{det} x)^{1-t}(\operatorname{det} y)^t .
$$

This construction is intrinsic to the fixed homogeneous realization $(\Omega,H,e)$. It uses the triangular group $H$ attached to the cone, rather than the ambient  geometry of $\Sym^{++}(3,\mathbb R)$.  
At the level of Cholesky factors, this is precisely the triangular geometric mean on the group  of lower triangular  matrices with positive diagonal entries, studied by Choi and Lim \cite{ChoiLim2022}.
  We give this mean a geometric interpretation on the sparse homogeneous cone and explain its relation to the cone's canonical metric.

This relation involves a distinction between metric connections and Levi-Civita connections. The left and right invariant parallelisms on $H$ induce two flat connections with nonzero torsion on $\Omega$, and the  Cholesky-Vinberg paths are geodesics for both. Each connection has a compatible metric obtained from a canonical structure on $V$. For the left invariant parallelism, that metric is exactly the canonical Hessian metric of $\Omega$. Nevertheless, the paths are generally not Levi-Civita geodesics for this metric, and their midpoints need not be its Riemannian midpoints. We prove the stronger statement that no $H$-invariant Riemannian metric has all these paths as affinely parametrized Levi-Civita geodesics. \\

The second construction  uses the logarithmic coordinates of the clan associated with the fixed realization $(\Omega, H, e)$. Its product is $u \Delta v=\underline{u} v+v \underline{u}^{\top}$,  
where  $\underline{u}$ is the lower triangular part of $u$ with half its diagonal entries. The clan exponential and logarithm are global inverse maps given by
$$
\ExpD(v)=\Phi(\exp \underline{v}),\qquad \LogD(x)=\log T_x+(\log T_x)^\top.
$$
 They define the \emph{Log-Vinberg path} 
$$
\gamma^{\Delta}_{x,y}(t)
=
\ExpD\bigl((1-t)\LogD(x)+t\LogD(y)\bigr),
\qquad 0\le t\le1,
$$
and the \emph{Log-Vinberg mean}
$$
x\lv y:=\gamma^{\Delta}_{x,y}(1/2).
$$
 Its geometry is flat and torsion-free and admits explicit weighted Fr\'echet  barycenters. 
 It  is covariance under clan automorphisms and translations in the logarithmic coordinates.  Moreover,  $\Omega$  endowed with the law
$$
x\oplus_{\Delta} y=\ExpD\bigl(\LogD(x)+\LogD(y)\bigr).
$$
is an abelian Lie group. The clan logarithm   a Lie group isomorphism from
$(\Omega,\oplus_{\Delta})$ onto $(V,+)$.\\

Both means studied here are symmetric, positively homogeneous and
preserve the defining sparsity, and both interpolate determinants
geometrically. Their covariance and geodesic properties express the
different roles of group multiplication and clan logarithmic
coordinates. \\

 The paper is organized as follows.
Section~\ref{Sec-Homg-Cones} recalls the part of Vinberg's theory
needed here, relating homogeneous convex cones to clans and
simply transitive solvable groups.
Section~\ref{Vinberg-cone} specializes this framework to the
five-dimensional Vinberg cone. We describe its triangular
group $H$, the Cholesky parametrization $\Phi$, and the
associated clan, and obtain explicit formulas for the
clan exponential and logarithm.

In Section~\ref{CV-mean}, we introduce the Cholesky-Vinberg
mean  and relate it to the geometric mean of lower
triangular matrices. We establish its main properties,
including $H$-equivariance, geometric interpolation of the
determinant, and the recursion identity for intermediate
points.
Section~\ref{Hessian} compares this construction with the
canonical Hessian metric. We show that the Cholesky-Vinberg
paths are not generally Levi-Civita geodesics and that their
midpoints need not be Riemannian midpoints for this metric.
Section~\ref{affine} then identifies two flat connections
with torsion for which these paths are geodesics. These
connections arise from the left and right invariant
parallelisms on $H$. We describe their compatible metrics,
identify one of them with the canonical Hessian metric,
and prove that no $H$-invariant Riemannian metric has all
the Cholesky-Vinberg paths as affinely parametrized
Levi-Civita geodesics.

Section~\ref{sec:log-fla-Lov-Vinberg} develops the Log-Vinberg
construction through the clan logarithmic coordinates.
We determine its geodesics and distance, prove completeness
and flatness of the resulting Riemannian metric, and derive
explicit formulas for weighted Fr\'echet means. We also
describe the abelian Lie group structure transported to
$\Omega$ by these coordinates.
Section~\ref{section:Lin-type} completes the comparison by examining the restriction of Lin's Log-Cholesky metric to $\Omega$. Section~\ref{sec:conclusion} concludes with some final remarks and directions for future work.

\section{Vinberg's theory of homogeneous convex cones}\label{Sec-Homg-Cones}

We briefly recall the part of Vinberg's theory that will be used later. Let $V$ be a finite-dimensional real vector space and let $\Omega\subset V$ be an open convex cone. The cone is called \emph{proper} if
$
\overline{\Omega}\cap(-\overline{\Omega})=\{0\},
$
and \emph{homogeneous} if its linear automorphism group
$$
G(\Omega):=\{g\in \GL(V):g\Omega=\Omega\}
$$
acts transitively on $\Omega$.   A \emph{symmetric cone} is a proper open convex cone $\Omega\subset V$ that is homogeneous and self-dual with respect to some Euclidean inner product on $V$.
By the Koecher-Vinberg theorem, symmetric cones are exactly the cones of squares of invertible elements in Euclidean Jordan algebras, see \cite{FarautKoranyi1994,Vinberg1963, Koecher}.

Building on earlier work of Koszul \cite{Koszul} on affine and homogeneous geometry, Vinberg extended the Koecher-Vinberg correspondence from symmetric cones to arbitrary homogeneous convex cones by means of \emph{clans}, or compact normal left-symmetric algebras. If $\Omega$ is a proper homogeneous cone, then there exists a connected split solvable Lie subgroup $H\subset G(\Omega)$ acting simply transitively on $\Omega$; see \cite[Chapter~1, Theorem~1]{Vinberg1963}. Fix a base point $E\in\Omega$ and let $\mathfrak h$ be the Lie algebra of $H$. Differentiating the orbit map
$$
H\to \Omega,\qquad h\mapsto h\cdot E,
$$
at the identity yields a linear isomorphism
$$
\mathfrak h\to V,\qquad L\mapsto L\cdot E.
$$
Thus every $x\in V$ can be written uniquely in the form
$
x=L_x\cdot E
$
with $L_x\in\mathfrak h$, and this leads to the bilinear product
$$
x\triangle y:=L_x\cdot y,\qquad x,y\in V.
$$

With this product, $(V,\triangle)$ becomes a clan with unit element $E$. For instance, the associator
$$
[x\triangle y\triangle z]:=x\triangle(y\triangle z)-(x\triangle y)\triangle z
$$
is symmetric in the first two variables,
$$
[x\triangle y\triangle z]=[y\triangle x\triangle z],
$$
the \emph{Koszul bilinear form}
$$
(x\mid y):=\tr L_{x\triangle y}
$$
is positive definite, and every operator $L_x$ has only real eigenvalues.

Conversely, every clan with unit determines a homogeneous cone as the open orbit
$$
\Omega=\{(\exp L_x)\cdot E:x\in V\}.
$$
Vinberg's theorem therefore gives, up to isomorphism,  a correspondence between homogeneous convex cones and clans with unit; see \cite[Chapter~2, Theorem~2]{Vinberg1963}.  

\section{The five-dimensional Vinberg cone}\label{Vinberg-cone}
 
We now specialize Vinberg's general framework to the five-dimensional cone that will be studied throughout the paper. 

Let $\Sym(3,\R)$ denote the vector space of real symmetric $3\times 3$ matrices, and let $\Sym^{++}(3,\R)$ be the cone of real symmetric positive definite $3\times 3$ matrices. Consider the five-dimensional subspace
$$
V=\left\{x=\begin{pmatrix}
x_1&0&x_4\\
0&x_2&x_5\\
x_4&x_5&x_3
\end{pmatrix}: x_1,x_2,x_3,x_4,x_5\in\R\right\}\subset \Sym(3,\R),
$$
and the open cone of positive definite matrices in $V$,
$$
\Omega=V\cap \Sym^{++}(3,\R)=\{x\in V  : x\succ 0\}.
$$
Its $(1,2)$ and $(2,1)$ entries vanish by definition.
The positivity conditions are
$$
x_1>0,\qquad x_2>0,\qquad S(x):=x_3-\frac{x_4^2}{x_1}-\frac{x_5^2}{x_2}>0,
$$
where $S(x)$ is the Schur complement. We call $\Omega$ the (dual) \emph{Vinberg cone}; it is the smallest-dimensional homogeneous convex cone that is not symmetric.

We denote by $\hLie$ the Lie algebra  of lower triangular matrices, 
$$
\hLie=\left\{\begin{pmatrix}a_1&0&0\\0&a_2&0\\u_4&u_5&a_3\end{pmatrix}: a_1,a_2,a_3,u_4,u_5\in\R\right\}.
$$
Its connected and simply connected Lie group is
$$
H=\left\{\begin{pmatrix}a_1&0&0\\0&a_2&0\\u_4&u_5&a_3\end{pmatrix}: a_1,a_2,a_3>0,\ u_4,u_5\in\R\right\}.
$$
Then $H$ is a split solvable Lie group acting simply transitively on $\Omega$ by congruence,
$$
\rho(A)x=AxA^\top,\qquad A\in H,\quad x\in\Omega.
$$

For
$$
x=\begin{pmatrix}x_1&0&x_4\\0&x_2&x_5\\x_4&x_5&x_3\end{pmatrix}\in V
$$
we write
$$
\iota(x)=\underline{x}:=\begin{pmatrix}x_1/2&0&0\\0&x_2/2&0\\x_4&x_5&x_3/2\end{pmatrix},
$$
so that
$
x=\underline{x}+\underline{x}^{\top},
$
and the inverse of $\iota:V\to \hLie$ is the symmetrization map $A\mapsto A+A^{\top}$.

We use the standard clan product
\begin{equation}\label{clan-prod}
x\Delta y=\underline{x}\,y+y\,\underline{x}^{\top},\qquad x,y\in V,
\end{equation}
and write
$$
L_x(y)=x\Delta y.
$$
Then $(V,\Delta)$ is the clan with unit $e=I_3$ attached to $\Omega$.

\begin{proposition}
The map
\begin{equation}\label{chol}
\Phi:H\to\Omega,\qquad \Phi(T)=\rho(T)e=TT^{\top},
\end{equation}
is a diffeomorphism.
\end{proposition}

\begin{proof}
Write
$$
T=\begin{pmatrix}a_1&0&0\\0&a_2&0\\u_4&u_5&a_3\end{pmatrix}\in H.
$$
Then
\begin{equation}\label{PhiT}
\Phi(T)=\begin{pmatrix}a_1^2&0&a_1u_4\\0&a_2^2&a_2u_5\\a_1u_4&a_2u_5&u_4^2+u_5^2+a_3^2\end{pmatrix}\in\Omega.
\end{equation}
Conversely, if
$$
x=\begin{pmatrix}x_1&0&x_4\\0&x_2&x_5\\x_4&x_5&x_3\end{pmatrix}\in\Omega,
$$
its lower Cholesky factor is
\begin{equation}\label{Tx}
T_x=\begin{pmatrix}
\sqrt{x_1}&0&0\\
0&\sqrt{x_2}&0\\
\dfrac{x_4}{\sqrt{x_1}}&\dfrac{x_5}{\sqrt{x_2}}&\sqrt{S(x)}
\end{pmatrix}\in H,
\end{equation}
and $x=T_xT_x^{\top}$. Thus $\Phi$ is bijective. The expressions   \eqref{PhiT} and \eqref{Tx}  show that both $\Phi$ and its inverse $x\mapsto T_x$ are smooth on their domains. Hence $\Phi$ is a diffeomorphism.
\end{proof}
The decomposition \eqref{chol} is exactly the Cholesky decomposition of SPD matrices.

\begin{lemma}\label{29-april-1}
If $T\in H$ and $t\in\R$, then $T^t=\exp(t\log T)$ belongs to $H$. In particular, if
$$
T=\begin{pmatrix}a_1&0&0\\0&a_2&0\\u_4&u_5&a_3\end{pmatrix}\in H,
$$
then the principal square root is
\begin{equation}\label{square}
T^{1/2}=\begin{pmatrix}
\sqrt{a_1}&0&0\\
0&\sqrt{a_2}&0\\
\dfrac{u_4}{\sqrt{a_1}+\sqrt{a_3}}&\dfrac{u_5}{\sqrt{a_2}+\sqrt{a_3}}&\sqrt{a_3}
\end{pmatrix}.
\end{equation}
\end{lemma}

\begin{proof}
The spectrum of $T$ is contained in $(0,\infty)$, so its principal logarithm is defined. Holomorphic functional calculus preserves triangularity and the zero $(2,1)$-entry, hence $\log T\in\hLie$. Therefore $T^t=\exp(t\log T)\in H$ for every real $t$. For the square root, let
$$
S=\begin{pmatrix}b_1&0&0\\0&b_2&0\\v_4&v_5&b_3\end{pmatrix}.
$$
The equation $S^2=T$ gives
$$
b_1^2=a_1,\;\; b_2^2=a_2,\;\; b_3^2=a_3,\;\; 
(b_1+b_3)v_4=u_4,\;\;  (b_2+b_3)v_5=u_5.
$$
Choosing the positive square roots on the diagonal yields formula~\eqref{square}.
\end{proof}

\begin{proposition}
For $x,y\in V$, one has
$$
\exp(tL_x)y=\exp(t \underline{x})\,y\,\exp(t \underline{x})^{\top},\qquad t\in\R.
$$
In particular,
$$
\ExpD(x):=\exp(L_x)e=\exp(\underline{x})\exp(\underline{x})^{\top}=\Phi\bigl(\exp(\underline{x})\bigr).
$$
Hence the clan exponential $\ExpD:V\to\Omega$ is a diffeomorphism, with inverse, the clan logarithm,
$$
\LogD(x)=\log(T_x)+\log(T_x)^{\top}.
$$
\end{proposition}

\begin{proof}
For $x, y\in V$, from the clan structure we have
$$\exp(L_x)y=y+x\Delta y+\frac{1}{2}x\Delta(x\Delta y)+\cdots.$$
Fix $x,y\in V$ and set
$$
F(t)=\exp(t\underline{x})\,y\,\exp(t\underline{x})^{\top}.
$$
Then
$$
\begin{aligned}
F'(t)&=\underline{x}\exp(t\underline{x})\,y\,\exp(t\underline{x})^{\top}+\exp(t\underline{x})\,y\,\exp(t\underline{x})^{\top}\underline{x}^{\top}\\
&=\underline{x}F(t)+F(t)\underline{x}^{\top}\\
&=x\Delta F(t)=L_x(F(t)).
\end{aligned}
$$
Since $F(0)=y$, uniqueness of solutions of linear ODEs gives $F(t)=\exp(tL_x)y$. Setting $t=1$ and $y=e$ proves the formula for $\ExpD$.

Now $x\mapsto \underline{x}$ is a linear isomorphism from $V$ onto $\hLie$, the matrix exponential is a diffeomorphism from $\hLie$ onto $H$, and $\Phi$ is a diffeomorphism from $H$ onto $\Omega$. Therefore $\ExpD=\Phi\circ\exp\circ(x\mapsto\underline{x})$ is a diffeomorphism. The inverse is obtained by composing the inverse maps, thus
$$
\LogD(x)=(\log T_x)+(\log T_x)^{\top}.
$$
\end{proof}

Note that the equality $x=\ExpD(v)$ is equivalent to $T_x=\exp(\underline v)$ by uniqueness of the lower Cholesky factor.

 \begin{proposition}\label{prop:explicit-log-coordinates}
For
$$
 x=\begin{pmatrix}x_1&0&x_4\\0&x_2&x_5\\x_4&x_5&x_3\end{pmatrix}\in\Omega,
$$
set
$$
 a_1=\sqrt{x_1},\quad a_2=\sqrt{x_2},\quad a_3=\sqrt{S(x)}, \; \text{with } S(x)=x_3-\frac{x_4^2}{x_1}-\frac{x_5^2}{x_2}>0.
$$
Then $\LogD(x)=(v_1,\dots,v_5)\in V$ is given by
$$
v_1=\log x_1,
\qquad
v_2=\log x_2,
\qquad
v_3=\log S(x),
$$
$$
v_4=
\begin{cases}
\dfrac{x_4}{\sqrt{x_1}}\dfrac{\log(a_1/a_3)}{a_1-a_3}, & a_1\ne a_3,\\[1ex]
\dfrac{x_4}{x_1}, & a_1=a_3,
\end{cases}
\qquad
v_5=
\begin{cases}
\dfrac{x_5}{\sqrt{x_2}}\dfrac{\log(a_2/a_3)}{a_2-a_3}, & a_2\ne a_3,\\[1ex]
\dfrac{x_5}{x_2}, & a_2=a_3.
\end{cases}
$$
Conversely, if $v=(v_1,\dots,v_5)\in V$, define
$$
a_1=e^{v_1/2},
\qquad
 a_2=e^{v_2/2},
\qquad
 a_3=e^{v_3/2},
$$
$$
b_4=
\begin{cases}
2v_4\dfrac{a_1-a_3}{v_1-v_3}, & v_1\ne v_3,\\[1ex]
v_4a_1, & v_1=v_3,
\end{cases}
\qquad
b_5=
\begin{cases}
2v_5\dfrac{a_2-a_3}{v_2-v_3}, & v_2\ne v_3,\\[1ex]
v_5a_2, & v_2=v_3,
\end{cases}
$$
Then
$$
\ExpD(v)=\begin{pmatrix}
a_1^2&0&a_1b_4\\
0&a_2^2&a_2b_5\\
a_1b_4&a_2b_5&b_4^2+b_5^2+a_3^2
\end{pmatrix}.
$$
\end{proposition}

\begin{proof}
Write
$$
T_x=\begin{pmatrix}a_1&0&0\\0&a_2&0\\b_4&b_5&a_3\end{pmatrix},
\qquad
b_4=\frac{x_4}{\sqrt{x_1}},\qquad b_5=\frac{x_5}{\sqrt{x_2}}.
$$
For an analytic function $f$ and a triangular matrix of this form, the only possible nonzero off-diagonal entries of $f(T_x)$ are
$$
f(T_x)_{31}=b_4\frac{f(a_1)-f(a_3)}{a_1-a_3},
\qquad
f(T_x)_{32}=b_5\frac{f(a_2)-f(a_3)}{a_2-a_3},
$$
extended   by continuity when the diagonal arguments coincide. Taking $f=\log$ gives the stated formulas for $v_4$ and $v_5$, while the diagonal entries of $\LogD(x)=\log T_x+(\log T_x)^\top$ give $v_1,v_2,v_3$.

Conversely, the diagonal entries of $\exp(\underline v)$ are $a_1,a_2,a_3$, and the same divided difference formula with $f=\exp$ gives its $(3,1)$ and $(3,2)$ entries as $b_4$ and $b_5$. Applying $\Phi(T)=TT^\top$ yields the  expression for $\ExpD(v)$.
\end{proof}

\section{The Cholesky-Vinberg path and mean}\label{CV-mean}

\begin{definition}
For $x,y\in\Omega$ and $0\le t\le1$, the \emph{Cholesky-Vinberg path} from $x$ to $y$ is
$$
\gamma^{\mathrm{cv}}_{x,y}(t):=T_x[(T_x^{-1}T_y)^t][(T_x^{-1}T_y)^t]^\top T_x^\top=\Phi\bigl(T_x(T_x^{-1}T_y)^t\bigr).
$$
Its midpoint
$$
x\cv y: =\gamma^{\mathrm{cv}}_{x,y}(1/2)=T_x[(T_x^{-1}T_y)^{1/2}][(T_x^{-1}T_y)^{1/2}]^\top T_x^\top 
$$
is called the \emph{Cholesky-Vinberg mean} of $x$ and $y$.
\end{definition}

Lemma~\ref{29-april-1} shows that the principal power in the definition is well defined, because $T_x^{-1}T_y\in H$. 

The idempotency property $x\cv x=x$ is immediate. Also, the consistency with scalars holds,
$$
\diag(x_1,x_2,x_3)\cv\diag(y_1,y_2,y_3)=\diag(\sqrt{x_1y_1}, \sqrt{x_2y_2}, \sqrt{x_3y_3}).
$$

It is useful to interpret this path at the Cholesky factors  in terms of weighted geometric means on the triangular group. Accordingly, let $A,B\in H$ and $0\le t\le 1$. We define the weighted geometric mean
$$
A\gm_{H,t}B:=A^{1/2}\bigl(A^{-1/2}BA^{-1/2}\bigr)^tA^{1/2},
$$
where all powers are principal powers inside $H$. When $t=1/2$, we write the geometric mean
$$
A\gm_H B:=A\gm_{H,1/2}B.
$$
\begin{proposition}\label{prop:cv-choi-lim}
The group $H$ is stable under the weighted geometric mean. More precisely, for $A,B\in H$ and $0\le t\le 1$ one has
$
A\gm_{H,t}B\in H
$
and
$$
A\gm_{H,t}B=A(A^{-1}B)^t.
$$
In particular,
$$
A\gm_H B=A(A^{-1}B)^{1/2}.
$$
\end{proposition}

\begin{proof}
Because $H$ is closed under multiplication and inversion, $A^{-1/2}BA^{-1/2}\in H$. By Lemma~\ref{29-april-1}, its principal $t$th power also belongs to $H$, so the operator formula defines an element of $H$.

Since
$$
A^{-1/2}BA^{-1/2}=A^{1/2}(A^{-1}B)A^{-1/2},
$$
similarity invariance of principal powers gives
$$
\bigl(A^{-1/2}BA^{-1/2}\bigr)^t=A^{1/2}(A^{-1}B)^tA^{-1/2}.
$$
Multiplying by $A^{1/2}$ on both sides yields
$$
A\gm_{H,t}B=A(A^{-1}B)^t.
$$
This proves the assertions of the proposition.
\end{proof}

\begin{corollary}\label{Cor-Choi-Lim-Khalid}
For every $x,y\in\Omega$ and $0\le t\le 1$,
$$
\gamma^{\mathrm{cv}}_{x,y}(t)=\Phi\bigl(T_x\gm_{H,t}T_y\bigr).
$$
In particular,
$$
x\cv y=\Phi(T_x\gm_H T_y).
$$
\end{corollary}

\begin{proof}
By the proposition,
$$
T_x\gm_{H,t}T_y=T_x(T_x^{-1}T_y)^t.
$$
Applying $\Phi$ gives the result.
\end{proof}

\begin{corollary}\label{Cor2-Choi-Lim-Khalid}
Let $x, y\in\Omega$. The Cholesky factor of the Cholesky-Vinberg mean $x\cv y$ is the unique matrix $X\in H$ satisfying the Riccati equation
$$
XT_x^{-1}X=T_y.
$$
\end{corollary}

\begin{proof}
By the previous corollary,
$$
T_{x\cv y}=T_x\gm_H T_y=T_x(T_x^{-1}T_y)^{1/2}.
$$
Therefore
$$
T_{x\cv y}T_x^{-1}T_{x\cv y}=T_x(T_x^{-1}T_y)^{1/2}T_x^{-1}T_x(T_x^{-1}T_y)^{1/2}=T_y.
$$
Conversely, if $X\in H$ satisfies $XT_x^{-1}X=T_y$, then
$$
(T_x^{-1}X)^2=T_x^{-1}T_y.
$$
Since $T_x^{-1}X\in H$, uniqueness of the principal square root in $H$ gives
$$
T_x^{-1}X=(T_x^{-1}T_y)^{1/2},
$$
hence $X=T_x(T_x^{-1}T_y)^{1/2}=T_{x\cv y}$.
\end{proof}

\begin{remark}  
This interpretation connects the Cholesky-Vinberg mean directly with Choi and Lim's geometric mean on 
the group $\mathcal L_m$ of lower triangular $m\times m$ matrices with positive diagonal entries,
 where the operator formula, the integral formula, and the arithmetic-harmonic iteration coincide; see \cite{ChoiLim2022}. Since $H$ is a closed Lie subgroup of $\mathcal L_3$ and is stable under inversion, arithmetic and harmonic means, and principal powers, the iteration
$$
A_{k+1}=\frac{A_k+B_k}{2},
\qquad
B_{k+1}=\left(\frac{A_k^{-1}+B_k^{-1}}{2}\right)^{-1},
$$
started at $A_0=T_x$ and $B_0=T_y$, stays in $H$ and converges to $T_x\#_H T_y=T_{x\cv y}$.

\end{remark}

\begin{theorem}\label{thm-4-jul}
For $x,y\in\Omega$ and $0\le t\le1$, the path $\gamma^{\mathrm{cv}}_{x,y}(t)$ has the following properties.

\begin{itemize}
\item[(i)] Endpoints and reversal:
$$
\gamma^{\mathrm{cv}}_{x,y}(0)=x,\qquad \gamma^{\mathrm{cv}}_{x,y}(1)=y,\qquad \gamma^{\mathrm{cv}}_{y,x}(t)=\gamma^{\mathrm{cv}}_{x,y}(1-t).
$$
In particular $x\cv y=y\cv x$.

\item[(ii)] Positive homogeneity:
$$
\gamma^{\mathrm{cv}}_{\lambda x,\mu y}(t)=\lambda^{1-t}\mu^t\gamma^{\mathrm{cv}}_{x,y}(t),\qquad \lambda,\mu>0.
$$
Hence
$$
(\lambda x)\cv(\mu y)=\sqrt{\lambda\mu}\,(x\cv y).
$$

\item[(iii)] $H$-equivariance:
$$
\gamma^{\mathrm{cv}}_{hxh^{\top},hyh^{\top}}(t)=h\gamma^{\mathrm{cv}}_{x,y}(t)h^{\top},\qquad h\in H.
$$
In particular,
$$
(hxh^{\top})\cv(hyh^{\top})=h(x\cv y)h^{\top}.
$$

\item[(iv)] Determinant interpolation:
$$
\det\gamma^{\mathrm{cv}}_{x,y}(t)=\det(x)^{1-t}\det(y)^t.
$$
In particular,
$$
\det(x\cv y)=\sqrt{\det x\,\det y}.
$$

\end{itemize}
\end{theorem}

\begin{proof} 
For brevity set $K=T_x^{-1}T_y$.

For (i), the endpoint identities are immediate,
$$
\gamma^{\mathrm{cv}}_{x,y}(0)=\Phi(T_x)=x \;\; \text{and}\; \;
\gamma^{\mathrm{cv}}_{x,y}(1)=\Phi(T_xK)=\Phi(T_y)=y.
$$
Also
$$
\gamma^{\mathrm{cv}}_{y,x}(t)=\Phi\bigl(T_y(T_y^{-1}T_x)^t\bigr)=\Phi\bigl(T_xK\,K^{-t}\bigr)=\Phi\bigl(T_xK^{1-t}\bigr)=\gamma^{\mathrm{cv}}_{x,y}(1-t).
$$
Evaluating at $t=1/2$ gives the symmetry of the midpoint.

For (ii), note that $T_{\lambda x}=\sqrt{\lambda}\,T_x$ and $T_{\mu y}=\sqrt{\mu}\,T_y$. Hence
$$
T_{\lambda x}^{-1}T_{\mu y}=\left(\frac{\mu}{\lambda}\right)^{1/2}K,
$$
so
$$
\begin{aligned}
\gamma^{\mathrm{cv}}_{\lambda x,\mu y}(t)&=\Phi\left(\sqrt{\lambda}\,T_x\left(\left(\frac{\mu}{\lambda}\right)^{1/2}K\right)^t\right)\\
&=\Phi\left(\lambda^{1/2}\left(\frac{\mu}{\lambda}\right)^{t/2}T_xK^t\right)\\
&=\lambda^{1-t}\mu^t\Phi(T_xK^t)=\lambda^{1-t}\mu^t\gamma^{\mathrm{cv}}_{x,y}(t).
\end{aligned}
$$
The midpoint identity follows by setting $t=1/2$.

For (iii), because $h,T_x\in H$, the Cholesky factor of $hxh^{\top}$ is $hT_x$, and similarly the Cholesky factor of $hyh^{\top}$ is $hT_y$. Therefore
$$
T_{hxh^{\top}}^{-1}T_{hyh^{\top}}=T_x^{-1}T_y=K,
$$
and thus
$$
\gamma^{\mathrm{cv}}_{hxh^{\top},hyh^{\top}}(t)=\Phi\bigl(hT_xK^t\bigr)=h\Phi(T_xK^t)h^{\top}=h\gamma^{\mathrm{cv}}_{x,y}(t)h^{\top}.
$$

For (iv), since $\det\Phi(T)=(\det T)^2$,
$$
\det\gamma^{\mathrm{cv}}_{x,y}(t)=\det(T_x)^2\det(K)^{2t}.
$$
Now $\det(T_x)^2=\det x$ and
$$
\det K=\frac{\det T_y}{\det T_x}=\left(\frac{\det y}{\det x}\right)^{1/2}.
$$
Hence
$$
\det\gamma^{\mathrm{cv}}_{x,y}(t)=\det(x)\left(\frac{\det y}{\det x}\right)^t=\det(x)^{1-t}\det(y)^t.
$$
In particular,
$$
\det(x\cv y)=\sqrt{\det x\,\det y}.
$$
\end{proof}

Next we prove the segment property of the weighted Cholesky-Vinberg mean.
\begin{proposition}\label{recursion}
For $x,y\in\Omega$ and $0\le t\le 1$, set
$$
x\cv_{t} y:=\gamma_{x,y}^{\mathrm{cv}}(t)=\Phi\bigl(T_x(T_x^{-1}T_y)^t\bigr).
$$
Then for every $0\le s\le t\le 1$ and every $0\le \lambda\le 1$,
$$
\bigl(x\cv_s y\bigr)\cv_\lambda\bigl(x\cv_t y\bigr)
=
x\cv_{(1-\lambda)s+\lambda t} y.
$$
In particular,
$$
\bigl(x\cv_s y\bigr)\cv\bigl(x\cv_t y\bigr)
=
x\cv_{(s+t)/2}y,
$$
where $\cv=\cv_{1/2}$.
Hence
$$
x\cv_t y=
\begin{cases}
x\cv\bigl(x\cv_{2t}y\bigr), & 0\le t\le \frac12,\\[1ex]
\bigl(x\cv_{2t-1}y\bigr)\cv y, & \frac12\le t\le 1.
\end{cases}
$$
\end{proposition}

\begin{proof}
Let
$
K:=T_x^{-1}T_y.
$ 
Then
$$
x\cv_s y=\Phi(T_xK^s),
\qquad
x\cv_t y=\Phi(T_xK^t).
$$
Therefore
$$
(T_xK^s)^{-1}(T_xK^t)=K^{t-s}.
$$
It follows that
$$
\bigl(x\cv_s y\bigr)\cv_\lambda\bigl(x\cv_t y\bigr)
=
\Phi\bigl(T_xK^s(K^{t-s})^\lambda\bigr)
=
\Phi\bigl(T_xK^{\,s+\lambda(t-s)}\bigr)
=
x\cv_{(1-\lambda)s+\lambda t} y.
$$
Taking $\lambda=\frac12$ gives the midpoint identity.
The last recursive formula follows by choosing $(s,t)=(0,2t)$ when $0\le t\le\frac12$
and $(s,t)=(2t-1,1)$ when $\frac12\le t\le1$.
\end{proof}

\section{The canonical Hessian metric}\label{Hessian}
We refer the reader to \cite{FarautKoranyi1994,Terras1988} for a detailed
account of the Riemannian symmetric space structure of $\Sym^{++}(3,\mathbb R)$. The
Riemannian  geometric notions used in this section, and later in
Sections~\ref{affine}   and~\ref{sec:log-fla-Lov-Vinberg}, are standard; we follow the terminology and conventions of
\cite{John-Lee}.

We now compare the Cholesky-Vinberg construction with the canonical Riemannian geometry of the Vinberg cone. For the  Vinberg cone, the Koszul-Vinberg characteristic function is given explicitly by \cite[Eq.~(2.3)]{Ishi-Koufany},
$$
\varphi_{\Omega}(x):=\int_{\Omega^*} e^{-(x,y)}dy=c\,x_1^{1/2}x_2^{1/2}(\det x)^{-2},\qquad c>0,
$$
where $\Omega^*$ is the dual cone with respect to the trace pairing $(x, y)=\operatorname{tr}(x y)$. 
For $T\in H$ with diagonal $(a_1,a_2,a_3)$, the action on $V$ satisfies
$\det\rho(T)=a_1^3a_2^3a_3^4$. The transformation law
$\varphi_\Omega(\rho(T)e)=\varphi_\Omega(e)/\det\rho(T)$ and the Cholesky
formula also verify this expression directly. 
 Hence the canonical Hessian metric is \cite[Eq.~(5.2)]{Ishi-Koufany}
\begin{equation}\label{hess-metric}
\cg_x^{\mathrm{can}}(u,v)=D_uD_v\log\varphi_{\Omega}(x)=-\frac12\left(\frac{u_1v_1}{x_1^2}+\frac{u_2v_2}{x_2^2}\right)+2\tr(x^{-1}ux^{-1}v).
\end{equation}
 
This metric is invariant under the automorphism group $G(\Omega)$ of $\Omega$, and hence in particular under the simply transitive subgroup $H$.

For comparison, on $\Sym^{++}(3,\R)$ the  invariant Riemannian metric is
$$
\cg_x^{\mathrm{SPD}}(u,v)=\operatorname{tr}(x^{-1}ux^{-1}v).
$$
Thus, after restriction to $\Omega$, one has
$$
\cg_x^{\mathrm{can}}(u,v)
=
2\,\cg_x^{\mathrm{SPD}}(u,v)
-\frac12\left(\frac{u_1v_1}{x_1^2}+\frac{u_2v_2}{x_2^2}\right).
$$
In particular, the canonical metric on $\Omega$ is not simply the restriction of the ambient SPD metric. The correction term is one manifestation of the non-symmetric nature of $\Omega$.

\begin{theorem}\label{thm:not-lc-geodesic-can}
The Cholesky-Vinberg path is generally not a geodesic for the Levi-Civita connection of the canonical Hessian metric $\cg^{\mathrm{can}}$.
\end{theorem}

\begin{proof}
It is enough to produce one counterexample. Take
$$
A=\begin{pmatrix}0&0&0\\0&0&0\\a&0&0\end{pmatrix}\in\mathfrak h,\qquad a\ne0,
$$
and consider the Cholesky-Vinberg path starting from the identity,
$
\gamma(t)=\Phi(\exp(tA)).
$
Since $A^2=0$, we have $\exp(tA)=I+tA$, so
$$
\gamma(t)=\begin{pmatrix}1&0&at\\0&1&0\\at&0&1+a^2t^2\end{pmatrix}.
$$
In the clan coordinates $(x_1,x_2,x_3,x_4,x_5)$ on $V$,  
$$
x_1(t)=1,\qquad x_2(t)=1,\qquad x_3(t)=1+a^2t^2,\qquad x_4(t)=at,\qquad x_5(t)=0,
$$
we have
$$
\dot x_1(0)=\dot x_2(0)=\dot x_3(0)=\dot x_5(0)=0,\qquad \dot x_4(0)=a,
$$
and
$$
\ddot x_1(0)=\ddot x_2(0)=\ddot x_4(0)=\ddot x_5(0)=0,\qquad \ddot x_3(0)=2a^2.
$$

We now compute the single Christoffel symbol $\Gamma^1_{44}$ that will matter. From the metric formula one finds
$$
\cg_{11}(e)=\frac32,\qquad \partial_4\cg_{14}(e)=-4,\qquad \partial_1\cg_{44}(e)=-4,
$$
so, because $\cg^{11}(e)=2/3$ and $\cg_{14}(e)=0$,
$$
\Gamma^1_{44}(e)=\frac12\cg^{11}(e)\bigl(2\partial_4\cg_{14}(e)-\partial_1\cg_{44}(e)\bigr)=-\frac43.
$$
Therefore the first coordinate of the geodesic equation, $\nabla_{\dot{\gamma}(t)} \dot{\gamma}(t)=0$, at $t=0$ becomes
$$
\ddot x_1(0)+\Gamma^1_{44}(e)\dot x_4(0)^2=0-\frac43a^2\ne0.
$$
Thus $\gamma$ is not a geodesic for the Levi-Civita connection of $\cg^{\mathrm{can}}$.
\end{proof}

\begin{corollary}\label{cor:cv-not-midpoint}
In general, the Cholesky-Vinberg mean $x\cv y$ is not the Riemannian midpoint of $x$ and $y$
for the canonical Hessian metric.  
\end{corollary}

\begin{proof}
Take the matrix $A\in\mathfrak h$ constructed in the proof of
Theorem~\ref{thm:not-lc-geodesic-can}, and put
$$
\gamma(t)=\Phi(\exp(tA)).
$$
For this curve, the Levi-Civita geodesic equation fails at
$t=0$. Choose a strongly convex normal neighborhood $U$ of $e$
and $\varepsilon>0$ such that $\gamma([0,\varepsilon])\subset U$.
Put $y=\gamma(\varepsilon)$. Then
$$
\gamma^{\mathrm{cv}}_{e,y}(t)=\gamma(\varepsilon t),
\qquad 0\le t\le1.
$$

Suppose, for a contradiction, that for every $p,q\in U$,
the Cholesky-Vinberg mean $p\cv q$ agrees with the midpoint
of the unique minimizing Levi-Civita geodesic from $p$ to $q$.
Let $\sigma:[0,1]\to U$ be the affinely parametrized minimizing
Levi-Civita geodesic from $e$ to $y$.

We prove by induction on $n$ that
$$
\gamma^{\mathrm{cv}}_{e,y}\left(\frac{k}{2^n}\right)
=
\sigma\left(\frac{k}{2^n}\right),
\qquad 0\le k\le2^n.
$$
For $n=0$, this follows from the common endpoints.
Suppose the assertion holds at level $n$.
At level $n+1$, the even indices follow immediately from the
induction hypothesis. For an odd index $2j+1$, with
$0\le j<2^n$, Proposition~\ref{recursion} gives
$$
\gamma^{\mathrm{cv}}_{e,y}
\left(\frac{2j+1}{2^{n+1}}\right)
=
\gamma^{\mathrm{cv}}_{e,y}\left(\frac{j}{2^n}\right)
\cv
\gamma^{\mathrm{cv}}_{e,y}\left(\frac{j+1}{2^n}\right).
$$
By the induction hypothesis, the two points on the right are
$\sigma(j/2^n)$ and $\sigma((j+1)/2^n)$.
They belong to $U$, so our assumption identifies their
Cholesky-Vinberg mean with their Riemannian midpoint.
Since $\sigma$ is affinely parametrized and minimizing,
this midpoint is
$$
\sigma\left(\frac{2j+1}{2^{n+1}}\right).
$$
This completes the induction.

The dyadic points are dense in $[0,1]$, so continuity gives
$\gamma^{\mathrm{cv}}_{e,y}=\sigma$.
Thus $\gamma(\varepsilon t)$ is a Levi-Civita geodesic,
contradicting the failure of the geodesic equation for
$\gamma$ at $t=0$. Therefore the two midpoint constructions
cannot agree for every pair in $U$.
\end{proof}

The canonical and ambient metrics therefore do not give the geodesic interpretation of the Cholesky-Vinberg mean. 
This suggests looking for a geometry adapted to the action of the solvable groupe $H$.

Such structures are naturally provided by the simply transitive solvable group attached to the
cone and by the underlying left-symmetric  algebra, or clan. This aligns closely with the work 
  of Koszul, Milnor, and Vinberg on flat affine, left-invariant
geometries, see \cite{Shima2007,Shima1976,Burde2006, Milnor, Koszul,Vinberg1963}.  The construction of such metric flat connections will be considered in the next sections.

\section{Torsionful flat affine geometries}\label{affine} { 

In this section we describe two flat connections with torsion and the metrics naturally attached to them.

\subsection{The clan inner product and the differential of the orbit map}

Recall that for
$$
u=
\begin{pmatrix}u_1&0&u_4\\0&u_2&u_5\\u_4&u_5&u_3\end{pmatrix}\in V
$$
we write
$$
\underline u=
\begin{pmatrix}u_1/2&0&0\\0&u_2/2&0\\u_4&u_5&u_3/2\end{pmatrix},
$$
and the clan product is
$$
u\Delta w=\underline u\,w+w\,\underline u^{\top}.
$$

\begin{definition}
The canonical clan inner product on $V$ is
$$
Q_{\Delta}(u,v):=\tr\bigl(L_{u\Delta v}\bigr),\qquad u,v\in V,
$$
where
$
L_u:V\to V$ is the left multiplication operator $L_u(w)=u\Delta w$.

We transport it to $\hLie$ by
\begin{equation}\label{def-Q_h}
Q_{\hLie}(A,B):=Q_{\Delta}(A+A^{\top},B+B^{\top}),\qquad A,B\in\hLie.
\end{equation}
\end{definition}

\begin{proposition}\label{prop:QDelta-formula}
For
$$
u=
\begin{pmatrix}u_1&0&u_4\\0&u_2&u_5\\u_4&u_5&u_3\end{pmatrix},
\qquad
v=
\begin{pmatrix}v_1&0&v_4\\0&v_2&v_5\\v_4&v_5&v_3\end{pmatrix},
$$
one has
\begin{equation}\label{Qh}
Q_{\Delta}(u,v)=\frac32u_1v_1+\frac32u_2v_2+2u_3v_3+4u_4v_4+4u_5v_5.
\end{equation}
Moreover,
$$
Q_{\Delta}(u,v)=\cg^{\mathrm{can}}_e(u,v),
$$
where $\cg^{\mathrm{can}}$ is the canonical Hessian metric \eqref{hess-metric} of the cone.
\end{proposition}

\begin{proof}
Let $E_1,\dots,E_5$ be the coordinate basis of $V$. A direct computation gives
$$
L_u(E_1)=u_1E_1+u_4E_4,
\qquad
L_u(E_2)=u_2E_2+u_5E_5,
\qquad
L_u(E_3)=u_3E_3,
$$
$$
L_u(E_4)=2u_4E_3+\frac{u_1+u_3}{2}E_4,
\qquad
L_u(E_5)=2u_5E_3+\frac{u_2+u_3}{2}E_5.
$$
Hence
$$
\tr(L_u)=\frac32u_1+\frac32u_2+2u_3.
$$
Now the first, second, and third coordinates of $u\Delta v$ are
$$
(u\Delta v)_1=u_1v_1,
\qquad
(u\Delta v)_2=u_2v_2,
\qquad
(u\Delta v)_3=u_3v_3+2u_4v_4+2u_5v_5.
$$
Therefore
$$
Q_{\Delta}(u,v)=\tr(L_{u\Delta v})=\frac32(u\Delta v)_1+\frac32(u\Delta v)_2+2(u\Delta v)_3,
$$
which is the desired  formula.

To compare with the canonical metric, recall its explicit formula,
$$
\cg^{\mathrm{can}}_x(a,b)=-\frac12\left(\frac{a_1b_1}{x_1^2}+\frac{a_2b_2}{x_2^2}\right)+2\tr(x^{-1}ax^{-1}b).
$$
At the identity $e=I_3$, this becomes
$$
\begin{aligned}
\cg^{\mathrm{can}}_e(u,v)&=-\frac12(u_1v_1+u_2v_2)+2\tr(uv)\\
&=\frac32u_1v_1+\frac32u_2v_2+2u_3v_3+4u_4v_4+4u_5v_5\\
&=Q_{\Delta}(u,v).
\end{aligned}
$$
 \end{proof}

For later use we study the differential of the orbit map. For $T\in H$ and $U\in T_TH\simeq \hLie$, the differential of
$
\Phi:H\to\Omega$, $\Phi(T)=TT^{\top}
$
is
$$
d\Phi_T(U)=UT^{\top}+TU^{\top}.
$$
Indeed,
$$
\Phi(T+\varepsilon U)=TT^{\top}+\varepsilon\bigl(UT^{\top}+TU^{\top}\bigr)+O(\varepsilon^2).
$$

\begin{proposition}\label{diff-inverse}
Let $x=\Phi(T)\in\Omega$ and let $u\in T_x\Omega\simeq V$. Then
\begin{equation}\label{diff-i}
(d\Phi_T)^{-1}(u)=T\,\underline{\,T^{-1}uT^{-\top}\,}.
\end{equation}
Explicitly, if
$$
T=\begin{pmatrix}a_1&0&0\\0&a_2&0\\u_4&u_5&a_3\end{pmatrix}
\; \text{ and }\; 
u=\begin{pmatrix}y_1&0&y_4\\0&y_2&y_5\\y_4&y_5&y_3\end{pmatrix},
$$
then the unique matrix $A=(d\Phi_T)^{-1}(u)\in\hLie$ is
$$
A=\begin{pmatrix}\alpha_1&0&0\\0&\alpha_2&0\\\beta_4&\beta_5&\alpha_3\end{pmatrix}
$$
with
$$
\alpha_1=\frac{y_1}{2a_1},\qquad
\alpha_2=\frac{y_2}{2a_2},\qquad
\alpha_3=\frac{1}{a_3}\left(\frac{y_3}{2}-u_4\beta_4-u_5\beta_5\right),
$$
$$
\beta_4=\frac{y_4-u_4\alpha_1}{a_1},\qquad
\beta_5=\frac{y_5-u_5\alpha_2}{a_2}.
$$
\end{proposition}

\begin{proof}
The identity $d\Phi_T(U)=UT^{\top}+TU^{\top}$ shows that solving $d\Phi_T(A)=u$ amounts to solving
$$
AT^{\top}+TA^{\top}=u.
$$
Set $B=T^{-1}A\in\hLie$. Then $A=TB$, and the equation becomes
$$
B+B^{\top}=T^{-1}uT^{-\top}.
$$
Since $T^{-1}uT^{-\top}\in V$ and $\iota:v\mapsto \underline v$ is the inverse of the symmetrization map $B\mapsto B+B^{\top}$ from $\hLie$ onto $V$, we get
$$
B=\underline{\,T^{-1}uT^{-\top}\,},
$$
which proves the formula \eqref{diff-i}. The coordinate expressions follow by writing out this identity explicitly.
\end{proof}

For $A\in\hLie$, let
$$
E_A^{\ell}(T)=TA,
\qquad
E_A^{r}(T)=AT
$$
denote the left-invariant and right-invariant vector fields on $H$. Their pushforwards under $\Phi$ are the following frames on $\Omega$,
$$
d\Phi_T\bigl(E_A^{\ell}(T)\bigr)=T(A+A^{\top})T^{\top}
$$
and
$$
d\Phi_T\bigl(E_A^{r}(T)\bigr)=A\Phi(T)+\Phi(T)A^{\top}.
$$
If $v\in V\simeq T_e\Omega$, then the clan fundamental vector field is
$$
X_v(x)=\underline v\,x+x\,\underline v^{\top},\qquad x\in\Omega.
$$
Hence
$$
X_v\bigl(\Phi(T)\bigr)=d\Phi_T\bigl(E^r_{\underline v}(T)\bigr).
$$
So the clan frame on $\Omega$ is the pushforward of the right-invariant frame on $H$.

\subsection{The Cholesky flat connection and its metric}

We use the standard connection defined by the left-invariant parallelism on $H$.

 \begin{proposition}\label{lft-in-met-H}
There is a unique left-invariant affine connection $\nabla^{H,\ell}$ on $H$ such that
$$
\nabla^{H,\ell}{}_{E_A^{\ell}}E_B^{\ell}=0,\qquad A,B\in\hLie.
$$
This connection is flat, and its torsion is
$$
\operatorname{Tor}^{\nabla^{H,\ell}}\bigl(E_A^{\ell},E_B^{\ell}\bigr)=-E_{[A,B]}^{\ell}.
$$
Its geodesics are exactly the curves
$$
\Gamma_{T_0,A}^{\ell}(t)=T_0e^{tA},\qquad T_0\in H,\quad A\in\hLie.
$$
\end{proposition}
 
\begin{proof}
Choose a basis $A_1,\dots,A_5$ of $\hLie$ and write $E_i=E_{A_i}^{\ell}$. If $Y=\sum_jg_jE_j$, set
$$
\nabla^{H,\ell}{}_X Y:=\sum_{j=1}^5 X(g_j)\,E_j.
$$
This formula proves existence and uniqueness. The frame is parallel, hence the connection is left-invariant and flat. Since $[E_A^{\ell},E_B^{\ell}]=E_{[A,B]}^{\ell}$, its torsion is
$$
\operatorname{Tor}^{\nabla^{H,\ell}}\bigl(E_A^{\ell},E_B^{\ell}\bigr)=-E_{[A,B]}^{\ell}.
$$
Finally, a curve is geodesic precisely when its left-trivialized velocity $T(t)^{-1}\dot T(t)$ is constant. If this constant is $A$, then
$$
\dot T(t)=T(t)A.
$$
The solution with $T(0)=T_0$ is $T(t)=T_0e^{tA}$.
\end{proof}

The corresponding connection on the cone $\Omega$ is the pushforward
$$
\nabla^{\mathrm{cv}}_UV:=\Phi_*\Bigl(\nabla^{H,\ell}{}_{\Phi_*^{-1}U}\,\Phi_*^{-1}V\Bigr).
$$

\begin{theorem}\label{cho-geo}
The connection $\nabla^{\mathrm{cv}}$ on $\Omega$ is flat and has nonzero torsion. For every $x,y\in\Omega$, the Cholesky-Vinberg path
$$
\gamma^{\mathrm{cv}}_{x,y}(t)=\Phi\bigl(T_x(T_x^{-1}T_y)^t\bigr)
$$
is the unique $\nabla^{\mathrm{cv}}$-geodesic joining $x$ to $y$, and its midpoint is the Cholesky-Vinberg mean,
$$
x\cv y=\gamma^{\mathrm{cv}}_{x,y}(1/2).
$$
\end{theorem}

\begin{proof}
Flatness and torsion properties are transported by the diffeomorphism $\Phi$. Let
$$
A_{x,y}=\log(T_x^{-1}T_y)\in\hLie.
$$
Then the $\nabla^{H,\ell}$-geodesic through $T_x$ with constant  
velocity $A_{x,y}$ is
$$
\Gamma_{x,y}^{\ell}(t)=T_xe^{tA_{x,y}}=T_x(T_x^{-1}T_y)^t.
$$
Applying $\Phi$ gives
$$
\Phi\bigl(\Gamma_{x,y}^{\ell}(t)\bigr)=\gamma^{\mathrm{cv}}_{x,y}(t),
$$
which is therefore the $\nabla^{\mathrm{cv}}$-geodesic from $x$ to $y$.  
Because $\exp:\hLie\to H$ is a diffeomorphism, $A_{x,y}$ is uniquely determined by the endpoints; uniqueness of the geodesic follows.
\end{proof}

\begin{proposition}\label{prop-left-inv}
The natural metric on $H$ associated with $\nabla^{H,\ell}$ is the left-invariant metric
\begin{equation}\label{fl-in-met}
\cg^{H,\ell}_T(U,V)=Q_{\hLie}(T^{-1}U,T^{-1}V),\qquad T\in H, \;\; U, V \in T_T H.
\end{equation}
\end{proposition}

\begin{proof}
For left-invariant fields $E_A^\ell(T)=TA$ and $E_B^\ell(T)=TB$,
$$
\cg_T^{H,\ell}\bigl(E_A^\ell(T),E_B^\ell(T)\bigr)
=
Q_{\mathfrak h}(A,B).
$$
Thus the metric is left-invariant and has constant coefficients in the $\nabla^{H,\ell}$-parallel frame. Consequently $\nabla^{H,\ell}\cg^{H,\ell}=0$.
\end{proof}

\begin{remark}
If
$$
T=
\begin{pmatrix}
a_1&0&0\\
0&a_2&0\\
u_4&u_5&a_3
\end{pmatrix},
$$
then the left Maurer-Cartan form $T^{-1}dT$ gives the  left-invariant coframe
$$
\omega_1=\frac{da_1}{a_1},
\;
\omega_2=\frac{da_2}{a_2},
\;
\omega_3=\frac{da_3}{a_3},
\;
\omega_4=\frac{du_4}{a_3}-\frac{u_4}{a_1a_3}\,da_1,
\;
\omega_5=\frac{du_5}{a_3}-\frac{u_5}{a_2a_3}\,da_2.
$$
and
$$
\cg^{H,\ell}
=
6\omega_1^2+6\omega_2^2+8\omega_3^2+4\omega_4^2+4\omega_5^2.
$$
\end{remark}

\begin{proposition}\label{prop:cv-can-metric}
Let $\cg^{\mathrm{cv}}$ be the pushforward metric of $\cg^{H,\ell}$ to $\Omega$ through $\Phi$. 
We have
$$
\Phi^*\cg^{\mathrm{can}}=\cg^{H,\ell},
$$
and consequently
$$
\cg^{\mathrm{cv}}=\cg^{\mathrm{can}}.
$$
In particular,
$$
\nabla^{\mathrm{cv}}\cg^{\mathrm{can}}=0.
$$
\end{proposition}
\begin{proof}
The equivariance $\Phi(ST)=\rho(S)\Phi(T)$ and the $H$-invariance of $\cg^{\mathrm{can}}$ imply that $\Phi^*\cg^{\mathrm{can}}$ is left-invariant. It is therefore enough to compare it with $\cg^{H,\ell}$ at the identity. Since $d\Phi_I(A)=A+A^\top$, Definition~\eqref{def-Q_h} gives
$$
(\Phi^*\cg^{\mathrm{can}})_I(A,B)
=
\cg_e^{\mathrm{can}}(A+A^\top,B+B^\top).
$$
The right-hand side equals $Q_{\mathfrak h}(A,B)=\cg^{H,\ell}_I(A,B)$, hence
$$
\Phi^*\cg^{\mathrm{can}}=\cg^{H,\ell}.
$$
This is equivalent to $\cg^{\mathrm{cv}}=\cg^{\mathrm{can}}$. Formula~\eqref{diff-i} also yields
$$
\cg^{\mathrm{cv}}_x(u,v)
=
Q_{\mathfrak h}\bigl(\underline{\,T_x^{-1}uT_x^{-\top}\,},
\underline{\,T_x^{-1}vT_x^{-\top}\,}\bigr).
$$
Finally,  Proposition~\ref{prop-left-inv} implies
$$
\Phi^*(\nabla^{\mathrm{cv}}\cg^{\mathrm{cv}})
=
\nabla^{H,\ell}(\Phi^*\cg^{\mathrm{cv}})
=
\nabla^{H,\ell} \cg^{H,\ell}
=
0.
$$
Thus $\nabla^{\mathrm{cv}}\cg^{\mathrm{can}}=0$.
\end{proof}

\begin{remark}
The identity $\nabla^{\mathrm{cv}}\cg^{\mathrm{can}}=0$ expresses metric compatibility. It does not identify $\nabla^{\mathrm{cv}}$ with the Levi-Civita connection, since $\nabla^{\mathrm{cv}}$ has nonzero torsion.
\end{remark}

\begin{proposition}\label{prop:no-left-invariant-riemannian}
There is no left-invariant Riemannian metric on $H$ whose Levi-Civita geodesics
include all the affinely parametrized curves
$$
t\longmapsto T_0e^{tA},
\qquad T_0\in H,\quad A\in\mathfrak h.
$$
Consequently, no metric on $\Omega$ obtained by pushing forward a left-invariant
Riemannian metric on $H$ through $\Phi$ has all Cholesky-Vinberg paths as
affinely parametrized Levi-Civita geodesics.
\end{proposition}

\begin{proof}
Assume, by contradiction, that $\cg$ is a left-invariant Riemannian metric on
$H$ with the stated property.

We identify each $X\in\mathfrak h$ with the corresponding left-invariant vector field $E_X^\ell(T)=TX$, for $T\in H$.
Since $\cg$ is left-invariant, its Levi-Civita connection $\nabla^{LC}$ is also
left-invariant.

Let $X\in\mathfrak h$ and consider the one-parameter subgroup
$
\gamma_X(t)=e^{tX}.
$
Its velocity is
$
\dot\gamma_X(t)=e^{tX}X=E_X^\ell(\gamma_X(t)).
$
Thus $\gamma_X$ is a Levi-Civita geodesic if and only if
$$
\nabla^{LC}_{E_X^\ell}E_X^\ell=0
$$
along $\gamma_X$. Since the vector field
$
\nabla^{LC}_{E_X^\ell}E_X^\ell
$
is left-invariant, it is enough to test this equality at the identity. Hence
the geodesic condition is equivalent to
$$
\nabla^{LC}_{X}X=0.
$$

Let $\langle\cdot,\cdot\rangle$ be the inner product on $\mathfrak h$ induced by $\cg$ at the identity.
For left-invariant vector fields, the derivative terms in the Koszul formula vanish, and therefore
$$
\langle \nabla^{LC}_X X,Y\rangle=-\langle [X,Y],X\rangle.
$$
Thus  
$$
\langle [X,Y],X\rangle=0
\qquad\text{for all }Y\in\mathfrak h.
$$
Since this holds for every $X\in\mathfrak h$, polarization yields
$$
\langle [X,Y],Z\rangle+\langle Y,[X,Z]\rangle=0
\qquad\text{for all }X,Y,Z\in\mathfrak h.
$$
Thus $\operatorname{ad}_X$ is skew-adjoint for every $X$. Consider
$$
H_1=\begin{pmatrix}1&0&0\\0&0&0\\0&0&0\end{pmatrix},
\qquad
E_4=\begin{pmatrix}0&0&0\\0&0&0\\1&0&0\end{pmatrix}.
$$
Since $[H_1,E_4]=-E_4$, skew-adjointness gives
$$
0=\langle[H_1,E_4],E_4\rangle+\langle E_4,[H_1,E_4]\rangle
=-2\langle E_4,E_4\rangle,
$$
which contradicts positive definiteness.

For the second assertion, $\Phi$ is an isometry when $\Omega$ is equipped with a pushed-forward metric. The Cholesky-Vinberg path from $x=\Phi(T_0)$ to $y=\Phi(T_0e^A)$ is
$$
\gamma^{\mathrm{cv}}_{x,y}(t)
=
\Phi\bigl(T_0(T_0^{-1}T_0e^A)^t\bigr)
=
\Phi(T_0e^{tA}).
$$
If all such paths were affinely parametrized Levi-Civita geodesics, their pullbacks $T_0e^{tA}$ would have the same property on $H$, contrary to the first assertion.
\end{proof}

\begin{remark}
The final part of the proof of the first assertion follows also from more general results of Milnor:
if $\operatorname{ad}_X$ is skew-adjoint for every $X\in\mathfrak h$, then by Milnor's criterion \cite[p. 296-297]{Milnor}, the metric $\cg$ is bi-invariant. Again by Milnor's theorem, this would force the group $H$ to be isomorphic to the direct product of a compact group and a commutative group, which is impossible since $H$ is connected, solvable, and non-abelian.

\end{remark}

\subsection{The clan flat connection and its metric}

We next describe the right-invariant counterpart of the previous geometry and its intrinsic expression on the cone in terms of clan vector fields.

\begin{proposition}
There is a unique right-invariant affine connection $\nabla^{H,r}$ on $H$ such that
$$
\nabla^{H,r}{}_{E_A^{r}}E_B^{r}=0,\qquad A,B\in\hLie.
$$
This connection is flat, and its torsion is
$$
\operatorname{Tor}^{\nabla^{H,r}}\bigl(E_A^{r},E_B^{r}\bigr)=E_{[A,B]}^{r}.
$$
Its geodesics are exactly the curves
$$
\Gamma_{T_0,A}^{r}(t)=e^{tA}T_0,\qquad T_0\in H,\quad A\in\hLie.
$$
\end{proposition}

\begin{proof}
The proof is similar to the proof of Proposition~\ref{lft-in-met-H}.
\end{proof}

We now push $\nabla^{H,r}$ forward to $\Omega$,
$$
\nabla^{\mathrm{cl}}_UV:=\Phi_*\Bigl(\nabla^{H,r}{}_{\Phi_*^{-1}U}\,\Phi_*^{-1}V\Bigr).
$$
In intrinsic terms, $\nabla^{\mathrm{cl}}$ is characterized by the clan frame,
$$
\nabla^{\mathrm{cl}}_{X_u}X_v=0,\qquad u,v\in V.
$$

\begin{proposition}\label{pro-cho-clan}
The connection $\nabla^{\mathrm{cl}}$ is flat, and
$$
\operatorname{Tor}^{\nabla^{\mathrm{cl}}}(X_u,X_v)=X_{u\Delta v-v\Delta u},
\qquad u,v\in V.
$$
For every $v\in V$, the integral curves of the constant clan field $X_v$ are $\nabla^{\mathrm{cl}}$-geodesics and are given by
$$
\eta_{x,v}(t)=e^{t\underline v}xe^{t\underline v^{\top}},\qquad x\in\Omega.
$$
For every $x,y\in\Omega$, if $v_{x,y}\in V$ is defined by
$$
\underline{v_{x,y}}=T_x\log(T_x^{-1}T_y)T_x^{-1},
$$
then
$$
\eta_{x,v_{x,y}}(t)=\gamma^{\mathrm{cv}}_{x,y}(t).
$$
Hence the Cholesky-Vinberg path is also a $\nabla^{\mathrm{cl}}$-geodesic.
\end{proposition}

\begin{proof}
Because $\nabla^{\mathrm{cl}}_{X_u}X_v=0$ in the clan frame, the connection is flat. The clan identity gives $[\underline u,\underline v]=\underline{u\Delta v-v\Delta u}$. Since these fields are pushforwards of right-invariant fields on $H$,
$$
[X_u,X_v]=-X_{u\Delta v-v\Delta u}.
$$
The stated torsion formula follows from $\operatorname{Tor}(X_u,X_v)=-[X_u,X_v]$ and is nonzero in general. If we set
$$
\eta_{x,v}(t)=e^{t\underline v}xe^{t\underline v^{\top}},
$$
then
$$
\frac{d}{dt}\eta_{x,v}(t)=\underline v\eta_{x,v}(t)+\eta_{x,v}(t)\underline v^{\top}=X_v\bigl(\eta_{x,v}(t)\bigr),
$$
so $\eta_{x,v}$ is the integral curve of the constant field $X_v$. Since $\nabla^{\mathrm{cl}}_{X_v}X_v=0$, every such integral curve is a $\nabla^{\mathrm{cl}}$-geodesic.

Now fix $x,y\in\Omega$ and define $v_{x,y}$ as above. Since $\mathfrak h$ is stable under conjugation by elements of $H$, the matrix $T_x\log(T_x^{-1}T_y)T_x^{-1}$ again belongs to $\mathfrak h$, so $v_{x,y}$ is well defined. Then
$$
e^{t\underline{v_{x,y}}}=T_xe^{tA_{x,y}}T_x^{-1},\qquad\text{with }\, A_{x,y}=\log(T_x^{-1}T_y).
$$
Using $x=T_xT_x^{\top}$, we obtain
$$
\eta_{x,v_{x,y}}(t)=T_xe^{tA_{x,y}}e^{tA_{x,y}^{\top}}T_x^{\top}=\Phi\bigl(T_xe^{tA_{x,y}}\bigr)=\Phi\bigl(T_x(T_x^{-1}T_y)^t\bigr)=\gamma^{\mathrm{cv}}_{x,y}(t).
$$
Thus the Cholesky-Vinberg path is a $\nabla^{\mathrm{cl}}$-geodesic as well.
\end{proof}

\begin{proposition}
The natural metric on $H$ associated with $\nabla^{H,r}$ is the right-invariant metric
$$
\cg^{H,r}_T(U,V)=Q_{\hLie}(UT^{-1},VT^{-1}),\qquad T\in H.
$$
\end{proposition}

\begin{proof}
This is the right-invariant analogue of Proposition~\ref{prop-left-inv}. The metric has constant coefficients in the $\nabla^{H,r}$-parallel right-invariant frame, so $\nabla^{H,r}\cg^{H,r}=0$.

\end{proof}

Recall that
the clan vector fields are defined by
$$
X_v(x)=v \Delta x=\underline{v}x+x\underline{v}^{\top}=\left.\frac{d}{dt}\right|_{t=0}e^{t\underline{v}}xe^{t\underline{v}^{\top}}.
$$

\begin{proposition}\label{prop-metric-cl}
 The metric on $\Omega$ naturally associated with $\nabla^{\mathrm{cl}}$ is the unique metric $\cg^{\mathrm{cl}}$ satisfying
\begin{equation}\label{g-cl1}
\cg^{\mathrm{cl}}_x\bigl(X_u(x),X_v(x)\bigr)=Q_{\Delta}(u,v),\qquad x\in\Omega,\quad u,v\in V.
\end{equation}
In coordinates with respect to the clan frame, if
$
\xi=\sum_{i=1}^5 c_iX_{e_i}(x)$ and 
$\eta=\sum_{i=1}^5 d_iX_{e_i}(x),
$
then
\begin{equation}\label{g-cl2}
\cg^{\mathrm{cl}}_x(\xi,\eta)=\frac32c_1d_1+\frac32c_2d_2+2c_3d_3+4c_4d_4+4c_5d_5.
\end{equation}
Moreover,
$$
\nabla^{\mathrm{cl}}\cg^{\mathrm{cl}}=0.
$$
 
\end{proposition}

\begin{proof}
Because $v\mapsto X_v(x)$ is an isomorphism from $V$ onto $T_x\Omega$ for every $x\in\Omega$, either rule \eqref{g-cl1} or \eqref{g-cl2} determines a   Riemannian metric on $\Omega$. Since its coefficients are constant in the $\nabla^{\mathrm{cl}}$-parallel clan frame, one has $\nabla^{\mathrm{cl}}\cg^{\mathrm{cl}}=0$. At $x=e$, the clan frame coincides with the standard basis of $V$, and therefore
$$
\cg^{\mathrm{cl}}_e=Q_{\Delta}=\cg_e^{\mathrm{can}}.
$$
\end{proof}

Note that 
in general
$
\cg^{\mathrm{cl}}\neq \cg^{\mathrm{can}}.
$

\section{The logarithmic flat connection and the Log-Vinberg mean}\label{sec:log-fla-Lov-Vinberg}

\subsection{The clan logarithmic geometry and the Log-Vinberg mean}

Let $D^V$ denote the standard flat connection on the
vector space $V$, that is, the unique affine connection such that in linear coordinates
$$
D^V_{\partial_i}\partial_j=0
\qquad\text{for all }i,j.
$$
Similarly, let $D^{\mathfrak h}$ denote the standard flat connection on the vector space
$\mathfrak h$, characterized in linear coordinates by
$$
D^{\mathfrak h}_{\partial_i}\partial_j=0
\qquad\text{for all }i,j.
$$
Since
$
\iota:V\to\mathfrak h$,
$\iota(v)=\underline v,
$
is a linear isomorphism, it preserves affine lines and therefore identifies these two flat
connections,
$$
D^{\mathfrak h}=\iota_*D^V.
$$

Because $\hLie$ is an associative algebra of lower triangular matrices and every element of $H$ has positive diagonal, the matrix exponential is a global diffeomorphism
$$
\exp:\hLie\to H
$$
with inverse the principal logarithm.

\begin{definition}
Let $\nabla^{H,\log}$ be the affine connection on $H$ obtained by transporting the flat connection $D^{\mathfrak h}$ on $\hLie$ through the global chart
$$
\log:H\to\hLie.
$$
Explicitly,
$$
\nabla^{H,\log}{}_X Y=(\exp)_*\Bigl(D^{\mathfrak h}_{(\log)_*X}(\log)_*Y\Bigr).
$$
\end{definition}

\begin{proposition}
The connection $\nabla^{H,\log}$ is flat and torsion-free. Its geodesics are exactly the curves
$$
\Gamma_{T_0,T_1}^{\log}(t)=\exp\bigl((1-t)\log T_0+t\log T_1\bigr),\qquad T_0,T_1\in H.
$$
\end{proposition}

\begin{proof}
In the global coordinates $A=\log T\in\hLie$, the Christoffel symbols of $\nabla^{H,\log}$ vanish identically. Therefore the geodesic equation is
$$
\ddot A(t)=0,
$$
and the unique geodesic joining $T_0$ and $T_1$ is obtained by exponentiating the affine segment from $\log T_0$ to $\log T_1$.
\end{proof}

\begin{proposition}
The natural metric on $H$ associated with $\nabla^{H,\log}$ is
$$
\cg^{H,\log}_T(U,V)=Q_{\hLie}\bigl((d\log)_TU,(d\log)_TV\bigr).
$$
In the affine coordinates $A=\log T\in\hLie$, this metric is constant,
$$
\cg^{H,\log}=6\,dA_{11}^2+6\,dA_{22}^2+8\,dA_{33}^2+4\,dA_{31}^2+4\,dA_{32}^2.
$$
Therefore $\nabla^{H,\log}$ is the Levi-Civita connection of $\cg^{H,\log}$.
\end{proposition}

\begin{proof}
This is immediate from the definition. Indeed, in the coordinates $A=\log T$, the metric is the constant bilinear form $Q_{\hLie}$, so its Levi-Civita connection is the standard flat connection.
\end{proof}

Now transport this logarithmic geometry to $\Omega$. Since
$$
\ExpD=\Phi\circ\exp\circ\ \iota,
\qquad
\LogD=\iota^{-1}\circ\log\circ\Phi^{-1},
$$
we define the affine connection $\nabla^{\Delta}$ on $\Omega$ by
$$
\nabla^{\Delta}_UV:=(\ExpD)_*\Bigl(D^V_{(\LogD)_*U}(\LogD)_*V\Bigr).
$$

\begin{proposition}\label{prop:nablaDelta-pushforward}
The connection $\nabla^{\Delta}$ is exactly the pushforward of $\nabla^{H,\log}$ by $\Phi$,
$$
\nabla^{\Delta}=\Phi_*\nabla^{H,\log}.
$$
In particular, $\nabla^{\Delta}$ is flat and torsion-free.
\end{proposition}

\begin{proof}
Since
$
\iota\circ\LogD=\log\circ\Phi^{-1},
$
differentiation gives
$$
\iota_*(\LogD)_*=(\log)_*(\Phi^{-1})_*.
$$
Using also $D^{\mathfrak h}=\iota_*D^V$, we obtain
$$
\begin{aligned}
\nabla^{\Delta}_UV&=(\Phi\circ\exp\circ\iota)_*\Bigl(D^V_{(\LogD)_*U}(\LogD)_*V\Bigr)\\
&=\Phi_*(\exp)_*\Bigl(D^{\mathfrak h}_{(\log)_*(\Phi^{-1})_*U}(\log)_*(\Phi^{-1})_*V\Bigr)\\
&=\Phi_*\Bigl((\nabla^{H,\log})_{(\Phi^{-1})_*U}(\Phi^{-1})_*V\Bigr).
\end{aligned}
$$
which proves the claim.
\end{proof}

\begin{definition}
Define a Riemannian metric $\cg^{\Delta}$ on $\Omega$ by
$$
\cg^{\Delta}_x(u,v):=Q_{\Delta}\bigl((d\LogD)_x u,(d\LogD)_x v\bigr),\qquad x\in\Omega,\quad u,v\in T_x\Omega.
$$
In other words, $\cg^{\Delta}$ is the pullback of the constant metric $Q_{\Delta}$ on $V$ through the global chart $\LogD$.
\end{definition}

For any $u\in V$, define 
$$\|u\|_{\Delta}:=Q_{\Delta}(u,u)^{1/2},\qquad u\in V.
$$

\begin{proposition}\label{prop:distance-log-vinberg}
The Riemannian distance of the metric $\cg^{\Delta}$ is
$$
d_{\Delta}(x,y)=\bigl\|\LogD(x)-\LogD(y)\bigr\|_{\Delta},\qquad x,y\in\Omega.
$$
Moreover, $(\Omega,\cg^\Delta)$ is complete, and the Levi-Civita connection of $\cg^{\Delta}$ is the flat torsion-free connection $\nabla^{\Delta}$.
\end{proposition}

\begin{proof}
By construction, the chart $\LogD$ identifies $(\Omega,\cg^{\Delta})$ isometrically with the Euclidean space $(V,Q_{\Delta})$. The distance formula, completeness, and the assertion about the Levi-Civita connection therefore follow from their Euclidean counterparts.
\end{proof}

\begin{definition}
For $x,y\in\Omega$ and $0\le t\le 1$, the \emph{Log-Vinberg path} is
$$
\gamma^{\Delta}_{x,y}(t)=\ExpD\bigl((1-t)\LogD(x)+t\LogD(y)\bigr).
$$
Its midpoint
$$
x\lv y:=\gamma^{\Delta}_{x,y}(1/2)
$$
is called the \emph{Log-Vinberg mean}.
\end{definition}

\begin{proposition}\label{prop:log-vinberg-naturality}
Let $h:V\to V$ be an automorphism of the clan $(V,\Delta,e)$; that is,
$$
h(e)=e,
\qquad
h(u\Delta v)=h(u)\Delta h(v).
$$
Then $h(\Omega)=\Omega$ and
$$
\LogD(hx)=h\LogD(x),
\qquad
h\bigl(\gamma^\Delta_{x,y}(t)\bigr)=\gamma^\Delta_{h(x),h(y)}(t).
$$
Moreover, $h$ is an isometry of $(\Omega,\cg^\Delta)$.
\end{proposition}

\begin{proof}
The clan identity implies $L_{h(u)}=h L_u h^{-1}$. Hence
$$
h\ExpD(u)=h\exp(L_u)e
=\exp(L_{h(u)})e
=\ExpD(hu).
$$
Since $\ExpD(V)=\Omega$, this proves $h(\Omega)=\Omega$ and, after inversion, $\LogD(hx)=h\LogD(x)$. The covariance of the path follows by applying this identity to its definition. Finally,
$$
Q_\Delta(hu,hv)
=\tr L_{h(u\Delta v)}
=\tr\bigl(h L_{u\Delta v}h^{-1}\bigr)
=Q_\Delta(u,v),
$$
so the definition of $\cg^\Delta$ shows that $h$ is an isometry.
\end{proof}

\begin{theorem}\label{thm:nablaDelta-geodesics}
For every $x,y\in\Omega$, the curve
$$
\gamma^{\Delta}_{x,y}(t)=\ExpD\bigl((1-t)\LogD(x)+t\LogD(y)\bigr),
\qquad 0\le t\le 1,
$$
is the unique affinely parametrized $\nabla^{\Delta}$-geodesic joining $x$ to $y$ on $[0,1]$. It is also the unique $\cg^{\Delta}$-geodesic segment from $x$ to $y$. Consequently,
$$
x\lv y=\gamma^{\Delta}_{x,y}(1/2)
$$
is the midpoint of this $\cg^{\Delta}$-geodesic segment. 
Moreover,
$$
 d_{\Delta}(x,\gamma^{\Delta}_{x,y}(t))=t\,d_{\Delta}(x,y),
\qquad
d_{\Delta}(y,\gamma^{\Delta}_{x,y}(t))=(1-t)\,d_{\Delta}(x,y).
$$
In particular,
$$
d_{\Delta}(x,x\lv y)=d_{\Delta}(y,x\lv y)=\frac12\,d_{\Delta}(x,y).
$$
\end{theorem}

\begin{proof}
Proposition~\ref{prop:distance-log-vinberg} identifies $\nabla^{\Delta}$ as the Levi-Civita connection of $\cg^{\Delta}$.

Let $z:[0,1]\to\Omega$ be a $C^2$ curve and set
$
v(t):=\LogD(z(t)).
$ Then
$$
(d\LogD)_{z(t)}\bigl(\nabla^{\Delta}_{\dot z(t)}\dot z(t)\bigr)
=
D^V_{\dot v(t)}\dot v(t)
=
\ddot v(t).
$$
Hence $z(t)$ is an affinely parametrized $\nabla^{\Delta}$-geodesic precisely when
$$
\ddot v(t)=0.
$$
The endpoint conditions $z(0)=x$ and $z(1)=y$ become
$
v(0)=\LogD(x)$ and $v(1)=\LogD(y).
$
The unique solution of this  boundary value problem in $V$ is
$$
v(t)=(1-t)\LogD(x)+t\LogD(y).
$$
Applying $\ExpD$ gives
$$
z(t)=\ExpD\bigl((1-t)\LogD(x)+t\LogD(y)\bigr)
=\gamma^{\Delta}_{x,y}(t).
$$
This proves both existence and uniqueness of the affinely parametrized $\nabla^{\Delta}$-geodesic from $x$ to $y$.

Finally, evaluating the same segment at $t=1/2$ gives
$$
\gamma^{\Delta}_{x,y}(1/2)
=\ExpD\left(\frac{\LogD(x)+\LogD(y)}{2}\right)
=x\lv y,
$$
so $x\lv y$ is the midpoint of the $\cg^{\Delta}$-geodesic segment.
The distance identities follow from Proposition~\ref{prop:distance-log-vinberg},
$$
 d_{\Delta}(x,\gamma^{\Delta}_{x,y}(t))
=
\bigl\|t(\LogD(y)-\LogD(x))\bigr\|_\Delta
=
t\,d_{\Delta}(x,y),
$$
  similarly
$$
 d_{\Delta}(y,\gamma^{\Delta}_{x,y}(t))=(1-t)\,d_{\Delta}(x,y).
$$

\end{proof}

\begin{remark}
The connection $\nabla^{\Delta}$ is different from both $\nabla^{\mathrm{cv}}$ and $\nabla^{\mathrm{cl}}$. In general, the geodesics $\gamma^{\Delta}_{x,y}$ and $\gamma^{\mathrm{cv}}_{x,y}$ are different. If $T_x$ and $T_y$ commute, then
$$
\gamma^{\Delta}_{x,y}(t)=\gamma^{\mathrm{cv}}_{x,y}(t)
$$
for all $t\in[0,1]$. In particular, on the diagonal subcone one has
$$
x\cv y=x\lv y.
$$
\end{remark}

The next proposition gives the formal properties that come directly from the logarithmic model and closely parallel the classical log-Euclidean case.

For $x,y\in\Omega$ and $t\in[0,1]$, set
$$x\odot_{\Delta,t} y:=\gamma^{\Delta}_{x,y}(t)=\ExpD\bigl((1-t)\LogD(x)+t\LogD(y)\bigr).$$

\begin{proposition}\label{prop:basic-log-vinberg-properties}
For all $x,y\in\Omega$ and $s,t, u\in[0,1]$, the following hold:
\begin{itemize}
\item[(1)]
$x\odot_{\Delta,0} y=x,$
$x\odot_{\Delta,1} y=y,$ 
$x\odot_{\Delta,t} x=x.
$
\item[(2)]
$x\odot_{\Delta,t} y=y\odot_{\Delta,1-t}x.
$
\item[(3)]
$\LogD(x\odot_{\Delta,t} y)=(1-t)\LogD(x)+t\LogD(y).
$
 \item[(4)]
$
\bigl(x\odot_{\Delta,s} y\bigr)\odot_{\Delta,t}\bigl(x\odot_{\Delta,u} y\bigr)=x\odot_{\Delta,(1-t)s+tu} y.
$
 \item[(5)] $h(x)\odot_{\Delta,t} h(y)=h(x\odot_{\Delta,t} y)$ for any automorphism of the clan $(V,\Delta,e)$.
 \end{itemize}
\end{proposition}

\begin{proof}
Properties (1) and (2) are immediate.

(3) Since $\ExpD$ maps $V$ onto $\Omega$, 
applying $\LogD$ to  $x\odot_{\Delta,t} y\in\Omega$  gives
$$
\LogD(x\odot_{\Delta,t} y)=(1-t)\LogD(x)+t\LogD(y),
$$
because $\LogD$ and $\ExpD$ are inverse maps.  

(4) follows by applying $\LogD$ and simplifying the affine combination in $V$. Property (5) follows from Proposition~\ref{prop:log-vinberg-naturality}.
\end{proof}

\begin{proposition}
\label{prop:homogeneity-log-vinberg}
For every $x\in\Omega$ and every $\lambda>0$,
$$
\LogD(\lambda x)=\LogD(x)+(\log\lambda)e,
$$
where $e=I_3$. Consequently, for all $\alpha,\beta>0$, all $x,y\in\Omega$, and all $0\le t\le 1$,
$$
(\alpha x)\odot_{\Delta,t}(\beta y)=\alpha^{1-t}\beta^t\,(x\odot_{\Delta,t}y).
$$
\end{proposition}

\begin{proof}
Let
$
u:=\LogD(x)\in V.
$
By definition,
$
\ExpD(u)=\exp(\underline{u})\,\exp(\underline{u})^\top.
$
Since $\underline e=\frac12 I_3$, for every $s\in\mathbb R$ we have
$$
\ExpD(u+se)
=
\exp\bigl(\underline{u}+\tfrac{s}{2}I_3\bigr)
\exp\bigl(\underline{u}+\tfrac{s}{2}I_3\bigr)^\top= e^s\ExpD(u).
$$
 Choosing
$
s=\log\lambda
$
gives
$$
\ExpD\bigl(u+(\log\lambda)e\bigr)=\lambda\ExpD(u)=\lambda x.
$$
Applying $\LogD$ yields
$$
\LogD(\lambda x)=\LogD(x)+(\log\lambda)e.
$$

Using this identity at both endpoints gives
$$
\LogD(\alpha x)=\LogD(x)+(\log\alpha)e,
\qquad
\LogD(\beta y)=\LogD(y)+(\log\beta)e.
$$
Hence
$$
\begin{aligned}
(\alpha x)\odot_{\Delta,t}(\beta y)
&=
\ExpD\Bigl((1-t)\LogD(\alpha x)+t\LogD(\beta y)\Bigr)\\
&=
\ExpD\Bigl((1-t)\LogD(x)+t\LogD(y)+\bigl((1-t)\log\alpha+t\log\beta\bigr)e\Bigr).
\end{aligned}
$$
Applying the first part once more, we obtain
$$
(\alpha x)\odot_{\Delta,t}(\beta y)
=
\exp\bigl((1-t)\log\alpha+t\log\beta\bigr)
\ExpD\bigl((1-t)\LogD(x)+t\LogD(y)\bigr).
$$
Therefore
$$
(\alpha x)\odot_{\Delta,t}(\beta y)=\alpha^{1-t}\beta^t\,(x\odot_{\Delta,t}y).
$$
\end{proof}

\begin{proposition}\label{prop:determinant-log-vinberg}
For every $u\in V$,
$$
\det\bigl(\ExpD(u)\bigr)=\exp\bigl(2\,\tr(\underline{u} )\bigr)=\exp\bigl(\tr( {u} )\bigr).
$$
Consequently, for all $x,y\in\Omega$ and $0\le t\le 1$,
\begin{equation}\label{det-interp}
\det(x\odot_{\Delta,t} y)=\det(x)^{1-t}\det(y)^t.
\end{equation}
In particular,
$$
\det(x\odot_{\Delta} y)=\sqrt{\det(x)\det(y)}.
$$
\end{proposition}

\begin{proof}
By definition,
$
\ExpD(u)=\exp(\underline{u})\exp(\underline{u})^{\top},
$
so
$$
\det\bigl(\ExpD(u)\bigr)=\det(\exp(\underline{u}))^2=\exp\bigl(2\tr(\underline{u})\bigr).
$$
If $u=\LogD(x)$ and $v=\LogD(y)$, then
$$
x\odot_{\Delta,t} y=\ExpD((1-t)u+tv),
$$
and linearity of $\iota$ and of the trace yields the determinant law \eqref{det-interp}.
\end{proof}

\subsection{Weighted Fr\'echet means and the abelian Lie group structure on   \texorpdfstring{$\Omega$}{Omega}}
We now interpret the Log-Vinberg mean as a Fréchet mean with respect to $d_\Delta$ and describe the abelian group structure on $\Omega$ induced by the clan logarithm.

\begin{theorem}\label{thm:weighted-log-vinberg-barycenter}
Let $x_1,\dots,x_N\in\Omega$, and let $w_1,\dots,w_N>0$ satisfy
$
\sum_{i=1}^N w_i=1.
$
Define
$$
G_w^\Delta(x_1,\dots,x_N)
:=
\ExpD\left(\sum_{i=1}^N w_i\,\LogD(x_i)\right).
$$
Then $G_w^\Delta(x_1,\dots,x_N)$ is the unique minimizer on $\Omega$ of the weighted Fr\'echet functional
$$
F(z):=\sum_{i=1}^Nw_i\,d_\Delta^2(z,x_i).
$$
That is,
$$
G_w^\Delta(x_1,\dots,x_N)
=
\operatorname*{arg\,min}_{z\in\Omega}
\sum_{i=1}^N w_i\,d_\Delta^2(z,x_i).
$$
Moreover,
$$
\det\bigl(G_w^\Delta(x_1,\dots,x_N)\bigr)
=
\prod_{i=1}^N \det(x_i)^{w_i}.
$$
\end{theorem}

\begin{proof}
Put $u_i=\LogD(x_i)$, $u=\LogD(z)$, and $\bar u=\sum_iw_iu_i$. Since $d_\Delta(z,x_i)=\|u-u_i\|_\Delta$ and $\sum_iw_i=1$, completion of the square gives
$$
F(z)=\|u-\bar u\|_\Delta^2+
\sum_{i=1}^Nw_i\|u_i\|_\Delta^2-\|\bar u\|_\Delta^2.
$$
The last two terms are independent of $u$, so $F(z)$ is minimized exactly when
$$
u=\bar u=\sum_{i=1}^N w_i\,\LogD(x_i).
$$
Proposition~\ref{prop:determinant-log-vinberg} and linearity of $u\mapsto\operatorname{tr}(\underline u)$ give
$$
\det\bigl(G_w^\Delta(x_1,\dots,x_N)\bigr)
=\exp\left(2\sum_{i=1}^Nw_i\operatorname{tr}(\underline{u_i})\right)
=\prod_{i=1}^N\det(x_i)^{w_i}.
$$
\end{proof}

\begin{corollary}\label{cor:explicit-log-barycenter-coordinates}
Write $\LogD(x_i)=(v_{i1},\ldots,v_{i5})$ and
$\bar v_j=\sum_iw_iv_{ij}$. Then
$$
G_w^\Delta(x_1,\ldots,x_N)=\ExpD(\bar v_1,\ldots,\bar v_5),
$$
whose entries are given by Proposition~\ref{prop:explicit-log-coordinates}
with $v_j=\bar v_j$, including its continuous extensions.  
\end{corollary}

\begin{definition}
For $x,y\in\Omega$, define
$$
x\oplus_{\Delta} y:=\ExpD\bigl(\LogD(x)+\LogD(y)\bigr).
$$
Define also the $\Delta$-inverse and the $\Delta$-powers by
$$
x^{-\Delta}:=\ExpD\bigl(-\LogD(x)\bigr),\quad
x^{\Delta,t}:=\ExpD\bigl(t\,\LogD(x)\bigr),\;\; t\in\mathbb R.
$$
\end{definition}

\begin{proposition}\label{prop:delta-group-law}
The operation $\oplus_{\Delta}$ makes $\Omega$ into an abelian Lie group, and the map $\LogD:(\Omega,\oplus_{\Delta})\to(V,+)$ is a Lie group isomorphism. Moreover, for all $a,x,y\in\Omega$ and $0\le t\le 1$,
$$
a\oplus_{\Delta}(x\odot_{\Delta,t} y)=(a\oplus_{\Delta}x)\odot_{\Delta,t}(a\oplus_{\Delta}y),
$$
$$
d_{\Delta}(a\oplus_{\Delta}x,a\oplus_{\Delta}y)=d_{\Delta}(x,y),
$$
and
$$
(x\odot_{\Delta,t} y)^{-\Delta}=x^{-\Delta}\odot_{\Delta,t} y^{-\Delta}.
$$
\end{proposition}

\begin{proof}
All assertions are immediate after applying $\LogD$, because $\LogD$ identifies $(\Omega,\oplus_{\Delta})$ with the additive group $(V,+)$ and identifies $x\odot_{\Delta,t}y$ with the affine interpolation of $\LogD(x)$ and $\LogD(y)$.
\end{proof}

\begin{remark}
There are thus three natural metric pictures on the Vinberg cone:
\begin{align*}
 (\nabla^{\mathrm{cv}},\cg^{\mathrm{can}})  &\longleftrightarrow (\nabla^{H,\ell},\cg^{H,\ell})   \\
 (\nabla^{\mathrm{cl}},\cg^{\mathrm{cl}})  &\longleftrightarrow  (\nabla^{H,r},\cg^{H,r})  \\
 (\nabla^{\Delta},\cg^{\Delta}) &\longleftrightarrow  (\nabla^{H,\log},\cg^{H,\log})   
\end{align*}
The first two are metric connections with torsion, while the third is torsion-free and is the Levi-Civita connection of a flat metric.
\end{remark}

\section{Restriction of Lin's Log-Cholesky metric to
\texorpdfstring{$\Omega$}{Omega}}\label{section:Lin-type} 

We now restrict Lin's Log-Cholesky metric to $H$ and transport
it to $\Omega$. For
$$
T=\begin{pmatrix}
a_1&0&0\\
0&a_2&0\\
u_4&u_5&a_3
\end{pmatrix}\in H,
$$
define
$$
\Psi(T)=(\log a_1,\log a_2,\log a_3,u_4,u_5).
$$
By \cite[Section~3.1]{Lin2019}, the restricted metric is
$$
\cg_H^{\mathrm{LC}}
=\sum_{i=1}^3 d(\log a_i)^2+du_4^2+du_5^2.
$$
We denote by $\cg^{\mathrm{LC}}$ its transport to $\Omega$
through $\Phi(T)=TT^{\top}$. This is precisely the restriction
to $\Omega$ of the ambient Log-Cholesky metric.

The following proposition specializes Lin's geodesic and mean
formulas to our cone; see
\cite[Proposition~3 and Corollary~12]{Lin2019}.

\begin{proposition}\label{prop:lin-flat-metric-on-omega}
The map $G=\Psi\circ\Phi^{-1}$ is a global isometry from
$(\Omega,\cg^{\mathrm{LC}})$ onto Euclidean space $\R^5$.
In particular, $\cg^{\mathrm{LC}}$ is flat and complete.
Moreover, $\Omega$ is totally geodesic in the ambient SPD cone
endowed with the Log-Cholesky metric.

Let $x,y\in\Omega$ have Cholesky factors
$$
T_x=\begin{pmatrix}
a_1&0&0\\
0&a_2&0\\
u_4&u_5&a_3
\end{pmatrix},
\qquad
T_y=\begin{pmatrix}
b_1&0&0\\
0&b_2&0\\
v_4&v_5&b_3
\end{pmatrix}.
$$
The unique affinely parametrized geodesic joining $x$ to $y$
on $[0,1]$ is
$$
\gamma^{\mathrm{LC}}_{x,y}(t)=\Phi(T(t)),
$$
where
$$
T(t)=
\begin{pmatrix}
a_1^{1-t}b_1^t&0&0\\
0&a_2^{1-t}b_2^t&0\\
(1-t)u_4+tv_4&(1-t)u_5+tv_5&a_3^{1-t}b_3^t
\end{pmatrix}.
$$
Its midpoint is
$$
x\#_{\mathrm{LC}}y=\Phi(T(1/2)).
$$
Both the geodesic and its midpoint preserve the prescribed
zero pattern of $\Omega$.
\end{proposition}

\begin{proof}
The map $\Psi$ is a global diffeomorphism from $H$ onto $\R^5$,
and $\cg_H^{\mathrm{LC}}$ is the pullback of the Euclidean
metric by $\Psi$. Thus $G=\Psi\circ\Phi^{-1}$ is a global
isometry, which proves flatness and completeness.

In the global Log-Cholesky coordinates on the full Cholesky
space, $H$ is the coordinate hyperplane defined by $T_{21}=0$.
Hence $H$ is totally geodesic, and so is its image $\Omega$
under the Cholesky isometry.

Finally, $T(t)\in H$ and
$$
\Psi(T(t))=(1-t)\Psi(T_x)+t\Psi(T_y).
$$
This gives the geodesic formula. Taking $t=1/2$ gives the
midpoint.
\end{proof}

\begin{remark}
Unlike the Cholesky-Vinberg mean, $\#_{\mathrm{LC}}$ is not
$H$-equivariant in general.

For example, let
$$
x=I_3,
\qquad
y=\begin{pmatrix}
4&0&2\\
0&9&0\\
2&0&5
\end{pmatrix},
\qquad
h=\begin{pmatrix}
1&0&0\\
0&2&0\\
1&0&1
\end{pmatrix}\in H.
$$
The midpoint formula gives
$$
\bigl((hxh^{\top})\#_{\mathrm{LC}}(hyh^{\top})\bigr)_{13}
=2\sqrt{2},
$$
whereas
$$
\bigl(h(x\#_{\mathrm{LC}}y)h^{\top}\bigr)_{13}
=2+\frac{\sqrt{2}}{2}.
$$
Thus the restricted Lin midpoint is not $H$-equivariant.

It also differs from the Log-Vinberg midpoint. For the same
$x$ and $y$, since $x=I_3$, one has
$x\lv y=\Phi(T_y^{1/2})$.  
Consequently,
$$
(x\lv y)_{13}
=\frac12
\ne
\frac{\sqrt{2}}2
=(x\#_{\mathrm{LC}}y)_{13}.
$$
\end{remark}

 \section{Final remarks}\label{sec:conclusion}

The Cholesky-Vinberg and Log-Vinberg means provide two ways
to interpolate on the   Vinberg cone while
preserving its prescribed zero pattern. Both interpolate
the determinant geometrically and therefore avoid determinant
swelling. Their geometric origins are different. The
Cholesky-Vinberg mean uses multiplicative interpolation in
the simply transitive triangular group $H$ and is equivariant
under its action. The Log-Vinberg mean uses affine interpolation
in the clan logarithmic coordinates. These coordinates give
$\Omega$ a complete flat Riemannian metric, explicit weighted
Fr\'echet means, and an abelian Lie group structure.

The relation with the canonical Hessian metric is more subtle.
The Cholesky-Vinberg paths are geodesics of two flat metric
connections with torsion, one of which is compatible with
this metric. They are not, in general, geodesics of its
Levi-Civita connection, and their midpoints need not coincide
with the corresponding Riemannian midpoints. 
Compatibility with the canonical Hessian metric does not,
by itself, make the Cholesky-Vinberg paths geodesics of
its Levi-Civita connection.

Lin's Log-Cholesky metric, restricted to $\Omega$, gives
another complete flat geometry with explicit means and the
same sparsity constraint. Its midpoint nevertheless differs
from the Log-Vinberg midpoint, and  
is generally not $H$-equivariant. The choice of geometry therefore
depends on more than preservation of the zero pattern:
the Cholesky-Vinberg construction retains the homogeneous
symmetry, while the logarithmic constructions give explicit
Riemannian barycenters.

The sparsity question also arises for general Kubo-Ando means.
One can show that the only ambient Kubo-Ando means whose
restrictions map $\Omega\times\Omega$ into $\Omega$ are the
weighted arithmetic means, see \cite{DK}. 

This obstruction concerns restriction, not transport.
The triangular operator means $Tf(T^{-1}S)$, with $T,S\in H$,
can instead be transported to $\Omega$ through $\Phi$.
Here $f:(0,\infty)\to(0,\infty)$ is operator monotone and
satisfies $f(1)=1$. The resulting operator means preserve the zero
pattern and are idempotent and $H$-equivariant. They are
symmetric when $f(t)=tf(t^{-1})$, but need not be monotone
in the L\"owner order. 
These questions   and applications  in models of quantum decoherence will be
studied in a forthcoming work, together with extensions of
the present constructions to more general homogeneous
convex cones.

\section*{Acknowledgments}
The author gratefully acknowledges support from \lq\lq Lorraine Université d'Excellence\rq\rq, part of the France 2030 program (ANR-15-IDEX-04-LUE).  He also acknowledge the support of the South Korea-France research cooperation PHC-STAR-2026 program. 
The author would like to thank the organizers of the   workshop  \lq\lq Operator  means and their applications --  Jeju, December
2025\rq\rq, where this project started.

\end{document}